\documentclass[10 pt, draft]{article}
\title{The Picard group of the moduli space of curves with level structures}
\author{Andrew Putman\footnote{Supported in part by NSF grant DMS-1005318}}
\usepackage{amsmath}
\usepackage{amssymb}
\usepackage{amsthm}
\usepackage{amscd}
\usepackage{epsfig}
\usepackage[margin=1.25in]{geometry}
\usepackage{amsfonts}
\usepackage[font=small,format=plain,labelfont=bf,up,textfont=it,up]{caption}
\usepackage{amscd}
\usepackage{pinlabel}
\usepackage{stmaryrd}
\usepackage{type1cm}
\usepackage{calc}
\usepackage{mathptmx}  %roman=Times, math=Times where possible

\theoremstyle{plain}
\newtheorem{theorem}{Theorem}[section]
\newtheorem{maintheorem}{Theorem}
\newtheorem{proposition}[theorem]{Proposition}
\newtheorem{lemma}[theorem]{Lemma}

\newtheorem{corollary}[theorem]{Corollary}

\newtheorem{step}{Step}

\newcommand\BeginSteps{\setcounter{step}{0}}

\theoremstyle{definition}
\newtheorem*{definition}{Definition}

\theoremstyle{remark}
\newtheorem*{remark}{Remark}

\DeclareMathOperator{\Hom}{Hom}
\DeclareMathOperator{\Map}{Map}

\DeclareMathOperator{\Ker}{ker}

\DeclareMathOperator{\Mod}{Mod}
\DeclareMathOperator{\Torelli}{{\mathcal I}}
\DeclareMathOperator{\Sp}{Sp}
\DeclareMathOperator{\SpLie}{\mathfrak{sp}}
\DeclareMathOperator{\GL}{GL}
\DeclareMathOperator{\SL}{SL}

\DeclareMathOperator{\Siegel}{\mathcal{H}}
\DeclareMathOperator{\Teich}{{\mathcal T}}
\DeclareMathOperator{\Moduli}{{\mathcal M}}
\DeclareMathOperator{\SpinModuli}{{\mathcal S}}

\DeclareMathOperator{\C}{\mathbb{C}}
\DeclareMathOperator{\Z}{\mathbb{Z}}
\DeclareMathOperator{\Q}{\mathbb{Q}}

\DeclareMathOperator{\ClosedField}{\overline{\mathbb{F}}}

\DeclareMathOperator{\HH}{H}

\DeclareMathOperator{\Trace}{Tr}
\DeclareMathOperator{\adj}{ad}
\DeclareMathOperator{\Image}{Im}

\newcommand\Span[1]{\ensuremath{\langle #1 \rangle}}
\newcommand\CaptionSpace{\hspace{0.2in}}

\newcommand\Figure[3]{
\begin{figure}[t]
\centering
\centerline{\psfig{file=#2,scale=60}}
\caption{#3}
\label{#1}
\end{figure}}

\DeclareMathOperator{\Tor}{tor}
\DeclareMathOperator{\Tf}{tf}
\DeclareMathOperator{\ContSheaf}{\mathcal{E}}
\DeclareMathOperator{\One}{\mathbb{I}}
\DeclareMathOperator{\Zero}{\mathbb{O}}
\DeclareMathOperator{\Pic}{Pic}
\DeclareMathOperator{\PicTriv}{\text{Pic}^0}
\DeclareMathOperator{\PicTop}{\text{Pic}_{\text{top}}}
\DeclareMathOperator{\PPAV}{\mathcal{A}}

\DeclareMathOperator{\Line}{\lambda}
\DeclareMathOperator{\Hodge}{\text{h}}
\newcommand{\PPAVLine}{\ensuremath{\lambda^{\text{a}}}}
\DeclareMathOperator{\StableRank}{sr}
\newcommand\MatTwoTwo[4]{\ensuremath
\left(\begin{smallmatrix}
#1 & #2 \\ #3 & #4
\end{smallmatrix}\right)}

\begin{document}

\maketitle

\begin{abstract}
For $4 \nmid L$ and $g$ large, we calculate the integral Picard
groups of the moduli spaces of curves and principally polarized abelian varieties with level $L$
structures.  In particular, we determine the divisibility properties of the standard line
bundles over these moduli spaces and we calculate the second integral cohomology group of the level $L$ subgroup
of the mapping class group (in a previous paper, the author determined this rationally).  
This entails calculating the abelianization of the level $L$ subgroup
of the mapping class group, generalizing previous results of Perron, Sato, and the author.  Finally,
along the way we calculate the first homology group of the mod $L$ symplectic group with
coefficients in the adjoint representation.
\end{abstract}

\section{Introduction}

Let $\Moduli_g$ be the moduli space of genus $g$ Riemann surfaces (see \cite{HarrisMorrison} for a survey).  
This is a quasiprojective orbifold whose
orbifold fundamental group is the {\em mapping class group} $\Mod_g$, that is, the group of isotopy classes
of orientation preserving diffeomorphisms of a closed orientable genus $g$ topological surface $\Sigma_g$ (see 
\cite{FarbMargalitBook} for a survey).  Denote
by $\Pic(\Moduli_g)$ the {\em Picard group} of $\Moduli_g$, that is, the set of algebraic line bundles on $\Moduli_g$ (as
an orbifold; see \S \ref{section:orbifolds} for the precise definition).  This forms an abelian group under tensor
products, and in the '60's Mumford \cite{MumfordMCG} showed that the first Chern class map gives an isomorphism
$\Pic(\Moduli_g) \cong \HH^2(\Mod_g;\Z)$.  Verifying a conjecture of Mumford, Harer \cite{HarerSecond} later proved
that $\HH^2(\Mod_g;\Z) \cong \Z$ for $g$ sufficiently large.

\paragraph{Level structures.}
Associated to any finite-index subgroup of $\Mod_g$ is a finite cover of $\Moduli_g$.  An important family
of such subgroups is obtained as follows.  Fix an integer $L \geq 2$.  The
group $\Mod_g$ acts on $\HH_1(\Sigma_g;\Z/L)$ and preserves the algebraic intersection form.  This is a nondegenerate
alternating form, so we obtain a representation $\Mod_g \rightarrow \Sp_{2g}(\Z/L)$.
The {\em level $L$ subgroup} of $\Mod_g$, denoted $\Mod_g(L)$, is the kernel of this representation.  The group
$\Mod_g(L)$ thus consists of all mapping classes that act trivially on $\HH_1(\Sigma_g;\Z/L)$.
The associated finite cover of $\Moduli_g$ is the {\em moduli space of curves with level $L$ structures}, denoted 
$\Moduli_g(L)$.
This is the space of pairs $(S,B)$, where $S$ is a genus $g$ Riemann surface and $B = \{a_1,b_1,\ldots,a_g,b_g\}$
is a {\em symplectic basis} for $\HH_1(S;\Z/L)$.  By a symplectic basis we mean that $B$ forms a basis for
$\HH_1(S;\Z/L)$ and 
$$i(a_i,b_j) = \delta_{i,j} \quad \text{and} \quad i(a_i,a_j)=i(b_i,b_j)=0$$ 
for $1 \leq i,j \leq g$, where $i(\cdot,\cdot)$ is the algebraic intersection pairing.

\paragraph{Picard groups of $\Moduli_g(L)$.}
The spaces $\Moduli_g(L)$ play an important role in the study of $\Moduli_g$.  One reason for this
is that for $L \geq 3$, the spaces $\Moduli_g(L)$ are fine moduli spaces rather than merely coarse ones.  The
key issue here is that Riemann surfaces can have automorphisms but Riemann surfaces with fixed level structures
cannot (see \cite[\S 0.2]{KockVainsencher} for a nice discussion of the difference between fine and coarse
moduli spaces).  It is
thus of some interest to determine their Picard groups.  
The first step in this direction was taken by Hain \cite{HainTorelli}.
A key step in Mumford's argument demonstrating that $\Pic(\Moduli_g) \cong \HH^2(\Mod_g;\Z)$
consisted in proving that $\HH^1(\Mod_g;\Z) = 0$.  Hain showed
that we also have $\HH^1(\Mod_g(L);\Z) = 0$ for $g \geq 3$.  Using this, he deduced that there is an injection
$\Pic(\Moduli_g(L)) \hookrightarrow \HH^2(\Mod_g(L);\Z)$.  In a recent paper \cite{PutmanSecond}, 
the author proved that $\HH^2(\Mod_g(L);\Z)$ has rank $1$ for $g \geq 5$.  

Our first main result is as follows.
For technical reasons, we will make the assumption that $4 \nmid L$, though we do not believe that
this hypothesis is necessary for our theorems.  See the end of the introduction for more details about this.

\begin{maintheorem}
\label{theorem:mg}
For $g \geq 5$ and $L \geq 2$ such that $4 \nmid L$, we have
$\Pic(\Moduli_g(L)) \cong \HH^2(\Mod_g(L);\Z)$.
\end{maintheorem}

\noindent
As a consequence of this and the aforementioned result of Hain asserting that
$\HH_1(\Mod_g(L);\Z)$ consists entirely of torsion, we have a short exact sequence
\begin{equation}
\label{eqn:mgpicexseq}
0 \longrightarrow \Hom(\HH_1(\Mod_{g}(L);\Z),\C^{\ast}) \longrightarrow \Pic(\Moduli_g(L)) \longrightarrow \Z \longrightarrow 0
\end{equation}
coming from the universal coefficients exact sequence for $\HH^2(\Mod_g(L);\Z)$.
The bulk of this paper is devoted to answering the following two
questions suggested by this calculation.
\begin{enumerate}
\item What is $\HH_1(\Mod_g(L);\Z)$?  For our answer, see Theorem \ref{theorem:modlabel1} below.  We remark
that determining this group is Problem 5.23 in \cite{FarbProblems}.  Combined with the author's
aforementioned computation of the rank of $\HH_2(\Mod_g(L);\Z)$, this gives a complete calculation
of $\HH^2(\Mod_g(L);\Z)$.
\item What specific line bundles give the various terms in \eqref{eqn:mgpicexseq}?  For our answer, see
Theorem \ref{theorem:mzgen} below.
\end{enumerate}

\paragraph{The standard line bundle on moduli space.}
We begin with the second question.  The elements of the torsion subgroup $\Hom(\HH_1(\Mod_{g}(L);\Z),\C^{\ast})$ come
from flat line bundles (see \S \ref{section:orbifolds}), so we focus on the $\Z$ term.
The first step is to describe a generator for $\Pic(\Moduli_g) \cong \Z$.  Let $\Hodge_g$ denote the
{\em Hodge bundle}.  This is the complex vector bundle over $\Moduli_g$ whose fiber
over the Riemann surface $S$ is the vector space of holomorphic $1$-forms on $S$.  By Riemann-Roch, $\Hodge_g$
is $g$-dimensional, so we have a complex line bundle $\Line_g = \wedge^g \Hodge_g$.  Arbarello-Cornalba
\cite{ArbarelloCornalba} showed that $\Pic(\Moduli_g) = \Span{\Line_g}$.

\paragraph{Divisibility.}
Let $\Line_g(L) \in \Pic(\Moduli_g(L))$ be the pullback of $\Line_g$.  It was recently shown
by G.\ Farkas \cite{Farkas} that for $L$ even, the line bundle $\Line_g(L)$ has a fourth root
modulo torsion.  More precisely, there exists some $\theta_g(L) \in \Pic(\Moduli_g(L))$
such that $4 \theta_g(L)-\Line_g(L)$ is torsion
(see the remark after Theorem \ref{theorem:farkas} below for how to extract this from Farkas's paper).  Our
next result shows that even modulo torsion, this is the best that one can do when $4 \nmid L$. 

\begin{maintheorem}[{Divisibility of $\Line_g(L)$}]
\label{theorem:mzgen}
For $g \geq 5$ and $L \geq 2$ such that $4 \nmid L$, the group
$\Pic(\Moduli_g(L))$ is generated modulo torsion by $\Line_g(L)$ if $L$ is odd and by
$\theta_g(L)$ if $L$ is even.
\end{maintheorem}

Our proof of Theorem \ref{theorem:mzgen} is topological.  The key result ends up being the following theorem.

\begin{maintheorem}[{Image of second homology group of $\Mod_g(L)$}]
\label{theorem:mh2}
For $g \geq 5$ and $L \geq 2$ such that $4 \nmid L$, the image of $\HH_2(\Mod_{g}(L);\Z)$ in
$\HH_2(\Mod_{g};\Z) \cong \Z$ is $n \Z$, where $n=1$ if $L$ is odd
and $n=4$ if $L$ is even.
\end{maintheorem}

\paragraph{Principally polarized abelian varieties.}
Let $\PPAV_g$ be the moduli space of principally polarized abelian varieties over $\C$ 
of dimension $g$ (see \S \ref{section:ppav}).  
This is a quasiprojective
orbifold with orbifold fundamental group $\Sp_{2g}(\Z)$, and there is a map $\Moduli_g \rightarrow \PPAV_g$
that takes a Riemann surface to its Jacobian.  At the level of fundamental groups, this corresponds
to the homomorphism $\Mod_g \rightarrow \Sp_{2g}(\Z)$ coming from the action of $\Mod_g$ on $\HH_1(\Sigma_g;\Z)$.
It is known that that the map $\Moduli_g \rightarrow \PPAV_g$ induces an isomorphism
$\Pic(\PPAV_g) \rightarrow \Pic(\Moduli_g)$ (see, e.g., \cite[Corollary 17.4]{HainModuliTran}).  
To prove Theorems \ref{theorem:mzgen} and \ref{theorem:mh2}, we must prove
analogous results for appropriate covers of $\PPAV_g$.

The {\em level $L$ subgroup of $\Sp_{2g}(\Z)$}, denoted $\Sp_{2g}(\Z,L)$, is the kernel of the map
$\Sp_{2g}(\Z) \rightarrow \Sp_{2g}(\Z/L)$.  Associated to $\Sp_{2g}(\Z,L)$ is the cover $\PPAV_g(L)$
of $\PPAV_g$, known as the {\em moduli space of principally polarized abelian varieties with level $L$ structures}.
Work of Kazhdan \cite{KazhdanT} shows that $\HH^1(\Sp_{2g}(\Z,L);\Z) = 0$ for $g \geq 2$.  Just
like for $\Moduli_g(L)$, this implies that there is an injection 
$\Pic(\PPAV_g(L)) \hookrightarrow \HH^2(\Sp_{2g}(\Z,L);\Z)$.  
Borel \cite{BorelStability1,BorelStability2} proved that $\HH^2(\Sp_{2g}(\Z,L);\Z)$ 
has rank $1$ for $g \geq 3$.  We prove the following.

\begin{maintheorem}
\label{theorem:ppav}
For $g \geq 4$ and $L \geq 2$ such that $4 \nmid L$, we have $\Pic(\PPAV_g(L)) \cong \HH^2(\Sp_{2g}(\Z,L);\Z)$.
\end{maintheorem}

\noindent
As a consequence of this and the aforementioned theorem of Borel asserting that
$\HH_1(\Sp_{2g}(\Z,L);\Z)$ consists entirely of torsion, we have an exact sequence
\begin{equation}
\label{eqn:ppavexseq}
0 \longrightarrow \Hom(\HH_1(\Sp_{2g}(\Z,L);\Z),\C^{\ast}) \longrightarrow \Pic(\PPAV_g(L)) \longrightarrow \Z \longrightarrow 0
\end{equation}
coming from the universal coefficients theorem for $\HH^2(\Sp_{2g}(\Z,L);\Z)$.

\begin{remark}
The group $\HH_1(\Sp_{2g}(\Z,L);\Z)$ has been calculated by Sato \cite{Sato} for $g \geq 3$
(for odd $L$, this was also determined independently 
by Perron \cite{Perron} and the author \cite{PutmanAbel}).  For $L$ odd, we have
$\HH_1(\Sp_{2g}(\Z,L);\Z) \cong \SpLie_{2g}(\Z/L)$ (see below for the definition
of the Lie algebra $\SpLie_{2g}(\Z/L)$), while for $L$ even,
$\HH_1(\Sp_{2g}(\Z,L);\Z)$ is an extension of $\SpLie_{2g}(\Z/L)$ by $\HH_1(\Sigma_g;\Z/2)$.
\end{remark}

\paragraph{Siegel modular forms.}
There is a line bundle $\PPAVLine_g \in \Pic(\PPAV_g)$ whose holomorphic sections are Siegel modular forms of 
weight $1$ and level $1$.  It is known that $\Pic(\PPAV_g) = \Span{\PPAVLine_g}$ and that 
$\PPAVLine_g$ pulls back to $\Line_g \in \Pic(\Moduli_g)$ (for instance, this follows from
\cite[\S 17]{HainModuliTran}).  
Let $\PPAVLine_g(L) \in \Pic(\PPAV_g(L))$ be the pullback of $\PPAVLine_g$.  For $n \in \Z$, holomorphic 
sections of $n \cdot \PPAVLine_g(L)$ are the Siegel modular forms of weight $n$ and level $L$.  For $L$ even, the classical
theta-nulls are Siegel modular forms of weight $1/2$ and level $L$ in an appropriate sense, so for $L$ even
we can divide $\PPAVLine_g(L)$ by $2$.  The following theorem says that if $4 \nmid L$, then even modulo torsion
this is the best one can do.

\begin{maintheorem}[{Divisibility of $\PPAVLine_g(L)$}]
\label{theorem:ppavzgen}
For $g \geq 4$ and $L \geq 2$ such that $4 \nmid L$, the group
$\Pic(\PPAV_g(L))$ is generated modulo torsion by $\frac{1}{n} \PPAVLine_g(L)$, where
$n = 1$ if $L$ is odd and $n=2$ if $L$ is even.
\end{maintheorem}

\begin{remark}
Work of Deligne \cite{Deligne} shows that $\PPAVLine_g(L)$ is never divisible by more than $2$.  However, one
cannot deduce from Deligne's work that $\PPAVLine_g(L)$ is not divisible by more than $2$ modulo torsion.
\end{remark}

\begin{remark}
It follows from work of Howe and Weissauer that all Siegel modular forms of sufficiently small weight come from theta
series (see \cite{FreitagSingular} for a precise formulation and history of this result).  From this, it follows 
that for $L$ odd
there do not exist Siegel modular forms of weight $1/2$ and level $L$.  One can view Theorem 
\ref{theorem:ppavzgen} as showing that this remains true even if we do not demand that our modular forms be holomorphic.
\end{remark}

Just like for $\Moduli_g$, our proof of Theorem \ref{theorem:ppavzgen} is topological and the key
step is to prove the following theorem.

\begin{maintheorem}[{Image of second homology group of $\Sp_{2g}(\Z,L)$}]
\label{theorem:splh2}
For $g \geq 4$ and $L \geq 2$ such that $4 \nmid L$, the image of $\HH_2(\Sp_{2g}(\Z,L);\Z)$ 
in $\HH_2(\Sp_{2g}(\Z);\Z) \cong \Z$ is $n \Z$, where $n=1$ if $L$ is odd
and $n=2$ if $L$ is even.
\end{maintheorem}

\paragraph{The adjoint representation of the symplectic group.}
A key technical result that goes into proving Theorem \ref{theorem:splh2}, which we believe
to be of independent interest, is as follows.  We first introduce some notation.
Denoting the $n \times n$ zero matrix by $\Zero_n$ and the $n \times n$
identity matrix by $\One_n$, let $\Omega_g$ be the matrix $\MatTwoTwo{\Zero_g}{\One_g}{-\One_g}{\Zero_g}$.
By definition, for a commutative ring $R$, the group $\Sp_{2g}(R)$ consists of $2g \times 2g$ matrices $X$ with entries in
$R$ that satisfy $X^t \Omega_{g} X = \Omega_g$.  We will denote by $\SpLie_{2g}(R)$ the additive group of all
$2g \times 2g$ matrices $A$ with entries in $R$ that satisfy $A^t \Omega_{g} + \Omega_g A = 0$.  It is easy to see
that $\Sp_{2g}(R)$ acts on $\SpLie_{2g}(R)$ by conjugation.

\begin{maintheorem}
\label{theorem:adjointcoho}
For $g \geq 3$ and $L \geq 2$ such that $4 \nmid L$, we have
$\HH_1(\Sp_{2g}(\Z/L);\SpLie_{2g}(\Z/L)) = 0$.
\end{maintheorem}

\begin{remark}
For $L$ a prime greater than $5$, this was originally proven by V\"{o}lklein \cite{Volklein}.
\end{remark}

\begin{remark}
For $L$ odd, the $\Sp_{2g}(\Z/L)$-representation $\SpLie_{2g}(\Z/L)$ is isomorphic to its dual 
(see \S \ref{section:symplecticlie}),
and it follows that $\HH^1(\Sp_{2g}(\Z/L);\SpLie_{2g}(\Z/L)) = 0$ as well.  This self-duality is false for $L$ even,
and in fact computer calculations indicate that $\HH^1(\Sp_{2g}(\Z/L);\SpLie_{2g}(\Z/L))$ can be nonzero for $L$ even.
We also remark that similar computer calculations suggest that the condition $4 \nmid L$ in Theorem \ref{theorem:adjointcoho}
is unnecessary.
\end{remark}

\begin{remark}
Our proof of Theorem \ref{theorem:adjointcoho} has a step that appeals to a computer calculation.  In the
end, this calculation comes down to determining the kernels and cokernels of rather large matrices
with entries in $\Z/p^k$ for small $p$ and $k$.  On a computer, large matrix calculations are 
frequently susceptible to overflow errors.  However, when working modulo $p^k$ it is 
straightforward to avoid this : one must merely be assiduous about reducing all 
numbers modulo $p^k$ after every arithmetic operation.  
\end{remark}

\paragraph{Torsion in the Picard group.}
Our final result is a calculation of the torsion in $\Pic(\Moduli_g(L))$, which 
is isomorphic to the dual of the abelian group $\HH_1(\Mod_g(L);\Z)$ (see \eqref{eqn:mgpicexseq} above).  
It is easiest to state our theorem for mapping class groups of surfaces with
boundary (see Theorem \ref{theorem:modlabel} for the closed case).  
Let $\Sigma_{g,1}$ denote a compact orientable genus $g$ topological surface with $1$ boundary component
and let $\Mod_{g,1}$ denote the group of isotopy classes of orientation preserving diffeomorphisms
of $\Sigma_{g,1}$ that act as the identity on $\partial \Sigma_{g,1}$.  Define $\Mod_{g,1}(L)$
to be the kernel of the action of $\Mod_{g,1}$ on $\HH_1(\Sigma_{g,1};\Z/L)$.

Our calculation of $\HH_1(\Mod_{g,1}(L);\Z)$ is closely related to some beautiful work of 
Dennis Johnson on the {\em Torelli group} $\Torelli_{g,1}$,
which is the kernel of the action of $\Mod_{g,1}$ on $\HH_1(\Sigma_{g,1};\Z)$.  Let $H = \HH_1(\Sigma_{g,1};\Z)$.
Johnson proved that
$$\HH_1(\Torelli_{g,1};\Z) \cong B_2(2g) \oplus \wedge^3 H,$$
where the two indicated terms are as follows (see \S \ref{section:torelli} for a more complete description).
\begin{itemize}
\item $\wedge^3 H$ comes from the {\em Johnson homomorphism}, which is a surjective homomorphism
$\tau : \Torelli_{g,1} \rightarrow \wedge^3 H$ arising from the action of $\Torelli_{g,1}$ on the second
nilpotent truncation of $\pi_1(\Sigma_{g,1};\Z)$.
\item $B_2(2g)$ is the space of boolean (i.e.\ square-free) polynomials over $\Z/2$ in $2g$ generators whose degree
is at most $2$.  These $2g$ generators correspond to a basis for $\HH_1(\Sigma_{g,1};\Z/2)$.  This portion of 
$\HH_1(\Torelli_{g,1};\Z)$
comes from work of Birman-Craggs \cite{BirmanCraggs} and Johnson \cite{JohnsonBirmanCraggs} 
and is related to the Rochlin invariant of homology $3$-spheres.  The corresponding abelian quotient of
$\Torelli_{g,1}$ is known as the {\em Birman-Craggs-Johnson homomorphism}.
\end{itemize}

We have a short exact sequence
$$1 \longrightarrow \Torelli_{g,1} \longrightarrow \Mod_{g,1}(L) \longrightarrow \Sp_{2g}(\Z,L) \longrightarrow 1.$$
To understand $\HH_1(\Mod_{g,1}(L);\Z)$, one must first investigate the extent to which
the Johnson and Birman-Craggs-Johnson homomorphisms extend over $\Mod_{g,1}(L)$.  The following results are
known (see \S \ref{section:extendingtorelli} for more details).
\begin{itemize}
\item Let $H_L = \HH_1(\Sigma_{g,1};\Z/L)$.  
A ``mod $L$'' Johnson homomorphism of the form $\Mod_{g,1}(L) \rightarrow \wedge^3 H_L$
was defined independently by Broaddus-Farb-Putman \cite{BroaddusFarbPutman},
by Perron \cite{Perron}, and by Sato \cite{Sato}.
\item Sato \cite{Sato} proved that ``most'' of the Birman-Craggs-Johnson homomorphism can be extended
over $\Mod_{g,1}(L)$ if $L$ is even.  More precisely, 
observe that $B_2(2g)$ contains a constant term $B_0(2g) \cong \Z/2$.  Sato proved that for $L$ even,
$\HH_1(\Mod_{g,1}(L);\Z)$ contains a term of the form $B_2(2g) / B_0(2g)$.  Sato also proved that
the $B_0(2g)$ piece of $\HH_1(\Torelli_{g,1};\Z)$ does {\em not} exist in $\HH_1(\Mod_{g,1}(2);\Z)$, and
asked whether a similar thing was true for general even $L$.
\end{itemize}
Our final theorem says that the above pieces together with $\HH_1(\Sp_{2g}(\Z,L);\Z)$ 
give the entire abelianization of $\Mod_{g,1}(L)$ if $4 \nmid L$.
It answers Sato's question in the affirmative in these cases.

\begin{maintheorem}[{Abelianization of $\Mod_g(L)$}]
\label{theorem:modlabel1}
Fix $g \geq 5$ and $L \geq 2$ such that $4 \nmid L$.  Setting $H_L = \HH_1(\Sigma_{g,1};\Z/L)$,
there is a short exact sequence
$$0 \longrightarrow K_{g,1} \longrightarrow \HH_1(\Mod_{g,1}(L);\Z) \longrightarrow \HH_1(\Sp_{2g}(\Z,L);\Z)
\longrightarrow 0,$$
where $K_{g,1} \cong \wedge^3 H_L$ if $L$ is odd and $K_{g,1} \cong (B_2(2g) / B_0(2g)) \oplus \wedge^3 H_L$ 
if $L$ is even.
\end{maintheorem}

\begin{remark}
For $L$ odd, this was previously and independently proven by the author \cite{PutmanAbel} and Sato \cite{Sato} and
by Perron \cite{Perron} modulo $2$-torsion.  Sato also proved it for $L = 2$.
\end{remark}

\begin{remark}
As we noted above, $\HH_1(\Sp_{2g}(\Z,L);\Z)$ is known (see \S \ref{section:splabel}).
\end{remark}

\begin{remark}
Fix $L$ such that $L$ is even but $4 \nmid L$.  It follows from the proof of Theorem \ref{theorem:modlabel1} that
the term $B_0(2g)$ of $\HH_1(\Torelli_{g,1};\Z)$ is in the kernel of the natural map 
$\HH_1(\Torelli_{g,1};\Z) \rightarrow \HH_1(\Mod_{g,1}(L);\Z)$.  As we shall see, this purely group-theoretic
vanishing phenomena is a consequence of the fact that $\Line_g(L) \in \Pic(\Moduli_g(L))$ is divisible 
modulo torsion by $4$ instead of merely $2$, like $\PPAVLine_g(L) \in \Pic(\PPAV_g(L))$.
\end{remark}

\begin{remark}
To remove the condition $4 \nmid L$ from the hypotheses of Theorems \ref{theorem:mg}--\ref{theorem:modlabel1} while
still using the basic techniques of this paper, one would have to remove the condition $4 \nmid L$ not 
only from the hypotheses of Theorem \ref{theorem:adjointcoho}, but also from the hypotheses of Theorems
\ref{theorem:stein} and \ref{theorem:hcoho} below.  Computer calculations indicate that this is probably
possible for Theorems \ref{theorem:adjointcoho} and \ref{theorem:hcoho}, but I am not sure if
it is possible for Theorem \ref{theorem:stein} (see the remark following Theorem \ref{theorem:stein} for more
details on this).
\end{remark}

\paragraph{Outline of paper and remarks.}
In \S \ref{section:preliminaries}, we will discuss some preliminary results and definitions about group cohomology
and orbifolds.  The proofs of Theorems \ref{theorem:ppavzgen}--\ref{theorem:adjointcoho} come next: in
\S \ref{section:symplecticgroup}, we discuss some basic results about the symplectic group, in 
\S \ref{section:twistedcoefficients}, we prove Theorem \ref{theorem:adjointcoho}, and in \S \ref{section:ppav}, we
introduce $\PPAV_g$ and prove Theorems \ref{theorem:ppavzgen} and \ref{theorem:splh2}.  In \S \ref{section:torelli},
we then discuss some results about the Torelli subgroup of the mapping class group needed for
the remaining theorems of the paper.  Finally, in \S \ref{section:modulispacecurves} we introduce $\Moduli_g$ and
prove Theorems \ref{theorem:mzgen}--\ref{theorem:mh2} and \ref{theorem:modlabel1}.

\paragraph{Notation.}
Throughout this paper, if $A$ is an abelian group, then we will denote by $A^{\Tor}$ the torsion
subgroup of $A$ and by $A^{\Tf}$ the quotient of $A$ by $A^{\Tor}$.

\paragraph{A naturality convention.}
Throughout this paper, the following will occur several times.  Let $\Gamma$ be a group, let $G$ be a normal
subgroup of $\Gamma$, and let $\phi : G \rightarrow M$ be a homomorphism, where $M$ is an abelian
group with a natural $\Gamma$ action.  For instance, $\Gamma$ might be $\Sp_{2g}(\Z)$ and $M$ might be
$\HH_1(\Sigma_g;\Z)$.  We will say that $\phi$ is {\em $\Gamma$-equivariant} if 
$\phi(y x y^{-1}) = y \cdot \phi(x)$ for all $x \in G$ and $y \in \Gamma$.

\paragraph{Acknowledgments.}
First of all, I wish to thank Richard Hain for suggesting that theorems like the ones proven in this paper
would be interesting, answering numerous algebro-geometric questions, and offering constant encouragement.  I
also wish to thank R.\ Keith Dennis, Gavril Farkas, Joe Harris, Brendan Hassett, 
Aaron Heap, Madhav Nori, Oscar Randal-Williams, Masatoshi Sato, 
and Michael Stein for useful conversations and correspondence.  Finally I wish to thank the referees for
their suggestions, which greatly improved the exposition of this paper.

\section{Preliminaries}
\label{section:preliminaries}

\subsection{Group cohomology}

We begin by reviewing some facts about group cohomology and establishing
some notation.  Good references include \cite{BrownCohomology} and \cite{Maclane}.

\paragraph{Degree zero.}
Let $G$ be a group and $M$ be a $G$-module.  The {\em invariants} of $M$, denoted $M^G$, is
the submodule $\{\text{$x \in M$ $|$ $g \cdot x = x$ for all $g \in G$}\}$.  The {\em coinvariants}
of $M$, denoted $M_G$, is the quotient $M/K$, where $K$ is the submodule spanned by the set 
$\{\text{$g \cdot x - x$ $|$ $x \in M$, $g \in G$}\}$.  We have $\HH^0(G;M) \cong M^G$ and
$\HH_0(G;M) \cong M_G$.

\paragraph{Duality.}
For $M$ an abelian group,
define $M^{\ast} = \Hom(M,S^1)$.  Observe that if $M$ is a
module over $\Z/L$, then $M^{\ast} \cong \Hom(M,\Z/L)$.  We have the following
duality between twisted homology and cohomology.

\begin{lemma}[{\cite[Proposition VI.7.1]{BrownCohomology}}]
\label{lemma:duality}
Let $G$ be a group and let $M$ be a $G$-module.  Then 
there is a natural isomorphism $\HH^k(G;M^{\ast}) \cong (\HH_k(G;M))^{\ast}$ for every $k \geq 0$.
\end{lemma}

\paragraph{The Hochschild-Serre spectral sequence.}
Let
$$1 \longrightarrow K \longrightarrow G \longrightarrow Q \longrightarrow 1$$
be a short exact sequence of groups and let $M$ be a $G$-module.  The homology Hochschild-Serre spectral
sequence is a first quadrant spectral sequence converging to $\HH_{\ast}(G;M)$ whose $E^2$ page is of the
form
$$E^2_{p,q} \cong \HH_p(Q;\HH_q(K;M)).$$
The edge groups of this spectral sequence have the following interpretations.
\begin{theorem}[{\cite[Proposition 10.2]{Maclane}}]
\label{theorem:edgegroups}
\mbox{}
\begin{itemize}
\item $E^{\infty}_{p,0}$, a subgroup of $E^2_{p,0} \cong \HH_p(Q;\HH_0(K;M)) \cong \HH_p(Q;M_K)$, is
equal to
$$\Image(\HH_p(G;M) \rightarrow \HH_p(Q;M_K)).$$
\item $E^{\infty}_{0,q}$, a quotient of $E^2_{0,q} \cong \HH_0(Q;\HH_q(K;M)) \cong (\HH_q(K;M))_Q$, is
isomorphic to
$$\Image(\HH_q(K;M) \rightarrow \HH_q(G;M)).$$
\end{itemize}
\end{theorem}
\noindent
A similar spectral sequence exists in cohomology.

Two standard consequences of the Hochschild-Serre spectral sequence are the 5-term
exact sequences
$$\HH_2(G;M) \longrightarrow \HH_2(Q;M_K) \longrightarrow (\HH_1(K;M))_Q \longrightarrow \HH_1(G;M) \longrightarrow \HH_1(Q;M_K) \longrightarrow 0$$
and
$$\HH^2(G;M) \longleftarrow \HH^2(Q;M^K) \longleftarrow (\HH^1(K;M))^Q \longleftarrow \HH^1(G;M) \longleftarrow \HH^1(Q;M^K) \longleftarrow 0.$$
Several times we will use the following consequence of the 5-term exact sequence in homology.

\begin{lemma}
\label{lemma:killchunk}
Let $G$ be a group and $M$ be a $G$-module.  Assume that $G \cong G_1 \oplus G_2$, that $G_2 < G$ acts
trivially on $M$, and that $\HH_1(G_2;\Z) = 0$.  Then $\HH_1(G;M) \cong \HH_1(G_1;M)$ and $\HH^1(G;M) \cong \HH^1(G_1;M)$.
\end{lemma}
\begin{proof}
An immediate consequence of the 5-term exact sequence in homology of the short exact sequence
$$1 \longrightarrow G_2 \longrightarrow G \longrightarrow G_1 \longrightarrow 1$$
with coefficients in $M$.
\end{proof}

\paragraph{Equivariant homology.}
To define Chern classes of line bundles over orbifolds, we will need some
basic results about equivariant homology.  A standard reference is
\cite[Chapter VII]{BrownCohomology}.  Throughout this section, $X$ is a CW-complex
and $G$ is a discrete group acting cellularly and properly discontinuously on $X$.

\begin{remark}
We will often apply the results here when $X$ is a smooth manifold and $G$ acts smoothly on $X$.
This is valid since a theorem of Illman \cite{Illman} says such a manifold can always be given
a $G$-invariant triangulation.
\end{remark}

We start with the following definition.

\begin{definition}
Let $EG$ be the universal cover of a fixed classifying space for $G$. The group $G$ acts
freely on $EG \times X$, and we will denote by $EG \times_G X$ the result of the {\em Borel construction}, i.e.\ the
quotient $(EG \times X)/G$.  We define
$$\HH^{\ast}_G(X;\Z) = \HH^{\ast}(EG \times_G Z;\Z) \quad \text{and} \quad \HH_{\ast}^G(X;\Z) = \HH_{\ast}(EG \times_G X;\Z).$$
\end{definition}

\begin{remark}
It is easy to see that $\HH^{\ast}_G(X;\Z)$ and $\HH_{\ast}^G(X;\Z)$ do not depend on the choice of $EG$.
\end{remark}

\begin{remark}
Our definition is different from but equivalent to the definition given in our reference \cite{BrownCohomology};
see \cite[Exercise VII.7.3]{BrownCohomology}.
\end{remark}

We will need the following three results about $\HH^{\ast}_G(X;\Z)$.  Recall that if $A$ is
an abelian group, then $A^{\Tor}$ denotes the torsion subgroup of $A$.

\begin{lemma}
\label{lemma:omnibuseqco}
Let $X$ be a CW complex and $G$ be a discrete group acting cellularly and properly discontinuously
on $X$.
\begin{enumerate}
\item If $X$ is $n$-connected, then $\HH^{k}_G(X;\Z) \cong \HH^{k}(G;\Z)$
and $\HH_{k}^G(X;\Z) \cong \HH_{k}(G;\Z)$ for $0 \leq k \leq n$.
\item For all $k \geq 1$, we have a short exact sequence
$$0 \longrightarrow \Hom((\HH_{k-1}^G(X;\Z))^{\Tor}, \C^{\ast}) \longrightarrow \HH^k_G(X;\Z) \longrightarrow \Hom(\HH_k^G(X;\Z),\Z) \longrightarrow 0.$$
\item If $H$ is a normal subgroup of $G$ that acts freely on $X$, then $\HH^k_G(X;\Z) \cong \HH^k_{G/H}(X/H;\Z)$
and $\HH_k^G(X;\Z) \cong \HH_k^{G/H}(X/H;\Z)$.
\end{enumerate}
\end{lemma}

The first conclusion is proven exactly like \cite[Proposition VII.7.3]{BrownCohomology},
the second conclusion is a translation of the universal coefficients theorem, and
the third conclusion is \cite[Exercise VII.7.5]{BrownCohomology}.

\subsection{Ordinary Picard groups}
\label{section:ordinarypicard}

In this section, we discuss Picard groups.  A good reference for this material (which
contains proofs of all results for which we do not give citations or proofs)
is \cite[\S III.4]{WellsComplex}.  Let $X$ be a topological space.  
To avoid technicalities, we will assume that $X$ is homeomorphic to a CW complex and that
all homology groups of $X$ are finitely generated.  

A {\em complex line bundle} on $X$ is a $1$-dimensional complex vector bundle on $X$.
The {\em topological Picard group} of $X$, denoted $\PicTop(X)$, is the set of isomorphism
classes of complex line bundles on $X$.  The set $\PicTop(X)$ is an abelian group under 
the operation of fiberwise tensor products.  The identity element is the trivial line bundle $X \times \C$.
A continuous map $\phi : Y \rightarrow X$ induces a homomorphism 
$\phi^{\ast} : \PicTop(X) \rightarrow \PicTop(Y)$.

With respect to some trivializing open cover, the transition functions for a complex line
bundle on $X$ are continuous functions taking values in $\GL_1(\C) = \C^{\ast}$.  
Letting $\ContSheaf^{\ast}$ be the sheaf of nonvanishing continuous $\C$-valued functions on $X$, 
we thus have an isomorphism
$$\PicTop(X) \cong \HH^1(X;\ContSheaf^{\ast}).$$
Letting $\ContSheaf$ be the sheaf of $\C$-valued continuous functions on $X$, 
there is a short exact sequence of sheaves
\begin{equation}
\label{eqn:sheafexseq1}
\begin{CD}
0 @>>> \Z @>>> \ContSheaf @>{\text{exp}}>> \ContSheaf^{\ast} @>>> 0.
\end{CD}
\end{equation}
Here $\Z$ is the sheaf of locally constant integer valued functions on $X$ and $\text{exp}$ maps the function
$f$ to the function $e^{2 \pi i f}$.
Since $\ContSheaf$ is a fine sheaf, we have $\HH^1(X;\ContSheaf) = \HH^2(X;\ContSheaf) = 0$.  The
long exact sequence associated to \eqref{eqn:sheafexseq1} thus degenerates to give the following lemma.

\begin{lemma}
\label{lemma:firstchern1}
If $X$ is a CW complex, then
$\PicTop(X) \cong \HH^1(X;\ContSheaf^{\ast}) \cong \HH^2(X;\Z)$.
\end{lemma}

\begin{remark}
The element of $\HH^2(X;\Z)$ associated to $\mathcal{L} \in \PicTop(X)$ is known as the {\em first
Chern class} of $\mathcal{L}$ and is denoted $c_1(\mathcal{L})$.  
\end{remark}

We now make the following definition.

\begin{definition}
If $X$ is a CW complex and $\C^{\ast}$ is the sheaf of locally constant functions on $X$ taking
values in $\C^{\ast}$, then a {\em flat line bundle} on $X$ is an element
of $\HH^1(X;\C^{\ast})$.
\end{definition}

\noindent
In other words, a flat line bundle is a line bundle determined by transition functions
which are locally constant.  There is a natural map of sheaves 
$\C^{\ast} \rightarrow \ContSheaf^{\ast}$, and
the line bundles that can be endowed with the structure of a flat line bundle are exactly the 
image of the resulting map
$\HH^1(X;\C^{\ast}) \rightarrow \HH^1(X;\ContSheaf^{\ast})$.
There is a short exact sequence of sheaves
$$\begin{CD}
0 @>>> \Z @>>> \C         @>{\exp}>>       \C^{\ast}             @>>> 0
\end{CD}$$
whose long exact sequence in cohomology contains the segment
$$\begin{CD}
\HH^1(X;\C) @>{\phi}>> \HH^1(X;\C^{\ast}) @>{c_1}>> \HH^2(X;\Z)
\end{CD}$$
We now analyze the maps in this exact sequence.  Recall that if $A$ is an abelian group, then
$A^{\Tor}$ denotes the torsion subgroup of $A$.
\begin{itemize}
\item The map $c_1$ is the map which returns the first Chern class of the flat line bundle associated to an
element of $\HH^1(X;\C^{\ast})$.  This can be seen by considering the commutative diagram
$$\begin{CD}
0 @>>> \Z @>>> \ContSheaf @>{\exp}>> \ContSheaf^{\ast} @>>> 0\\
@.    @AA{=}A     @AAA                  @AAA                   @.\\
0 @>>> \Z @>>> \C         @>{\exp}>> \C^{\ast}             @>>> 0
\end{CD}$$
of short exact sequences of sheaves.  
\item The image of $c_1$ is isomorphic to $\Hom((\HH_1(X;\Z))^{\Tor},\C^{\ast})$.   Indeed, we have
isomorphisms
$$\HH^1(X;\C) \cong \Hom(\HH_1(X;\Z),\C) \quad \quad \text{and} \quad \quad \HH^1(X;\C^{\ast}) \cong \Hom(\HH_1(X;\Z),\C^{\ast}),$$
and $\phi$ is the evident map arising from the projection $\C \stackrel{\exp}{\longrightarrow} \C^{\ast}$.
In other words, the image of $\phi$ consists of elements of
$\Hom(\HH_1(X;\Z),\C^{\ast})$ which can be lifted to elements of $\Hom(\HH_1(X;\Z),\C)$.  This
is exactly the set of homomorphisms which vanish on the direct summand $(\HH_1(X;\Z))^{\Tor}$
of $\HH_1(X;\Z)$.  It follows 
that the cokernel of $\phi$ (which is isomorphic to the image of $c_1$) is isomorphic to
$\Hom((\HH_1(X;\Z))^{\Tor},\C^{\ast})$.
\end{itemize}
The universal coefficients theorem implies that the torsion subgroup of $\HH^2(X;\Z)$ is isomorphic
to the abelian group $\Hom((\HH_1(X;\Z))^{\Tor},\C^{\ast})$.  The following result follows.

\begin{lemma}
\label{lemma:torsionpicard}
The torsion subgroup of $\Pic(X)$ is equal to the set of line bundles on $X$ that can be
endowed with the structure of a flat line bundle.
\end{lemma}

\begin{remark}
It follows from the above discussion that the isomorphism class (as a complex line bundle) of
the flat line bundle associated to $\phi \in \Hom(\HH_1(X;\Z),\C^{\ast})$ only depends on the restriction
of $\phi$ to $(\HH_1(X;\Z))^{\Tor}$.
\end{remark}

We now give an alternate construction of flat line bundles that we will need later.  Consider
$$\phi \in \Hom(\HH_1(X;\Z),\C^{\ast}) \cong \HH^1(X;\C^{\ast}).$$  
Define $\phi' : \pi_1(X) \rightarrow \C^{\ast}$ to be the composition of $\phi$ with the 
natural projection $\pi_1(X) \rightarrow \HH_1(X;\Z)$.
Let $\tilde{X}$ be the universal cover of $X$.  Define an action of $\pi_1(X)$ on the trivial 
bundle $\tilde{X} \times \C$ by the formula
$$g \cdot (x,z) = (g \cdot x, \phi'(g) z) \quad \quad (x \in \tilde{X}, z \in \C, g \in \pi_1(X)).$$
Defining $E_{\phi} = (\tilde{X} \times \C) / \pi_1(X)$, we have a natural projection map
$E_{\phi} \rightarrow \tilde{X} / \pi_1(X) = X$.  This is a complex line bundle on $X$ which
has a natural flat structure.  We then have the following theorem.  

\begin{lemma}
\label{lemma:alternateflat}
The flat line bundle $E_{\phi} \rightarrow X$ equals $\phi \in \HH^1(X;\C^{\ast})$ under the
identification of flat line bundles with elements of $\HH^1(X;\C^{\ast})$.
\end{lemma}

\noindent
We do not know a reference that proves this theorem in exactly this form.  However,
in \cite[Chapter II.3]{KobayashiDiff} it is proven that for $\phi \in \HH^1(X;\C^{\ast})$, there
exists some $\psi : \HH_1(X;\Z) \rightarrow \C^{\ast}$ such that the flat line bundle
$E_{\psi} \rightarrow X$ equals $\phi$ under the identification of flat line bundles
with elements of $\HH^1(X;\C^{\ast})$.  Tracing through the construction in \cite[Chapter II.3]{KobayashiDiff},
it is easy to see that in fact $\psi = \phi$.

Now assume that $X$ has the structure of a quasiprojective variety over $\C$.  
The {\em Picard group of $X$}, denoted $\Pic(X)$, is the set of isomorphism classes
of algebraic line bundles over $X$.  There is an evident map $\Pic(X) \rightarrow \PicTop(X)$,
and thus a first Chern class map $\Pic(X) \rightarrow \HH^2(X;\Z)$.  This map is in general 
neither injective nor surjective.  However, the following lemma implies that its image
contains $(\HH^2(X;\Z))^{\Tor}$.

\begin{lemma}
\label{lemma:flatalgebraic}
Let $X$ be a quasiprojective variety over $\C$ and let $\phi : \HH_1(X;\Z) \rightarrow \C^{\ast}$ be
a homomorphism whose image is finite.  Then the flat line bundle $E_{\phi} \rightarrow X$ is
algebraic.
\end{lemma}

\noindent
Lemma \ref{lemma:flatalgebraic} is folklore, but we do not know a reference for it so we
include a proof.

\begin{proof}[{Proof of Lemma \ref{lemma:flatalgebraic}}]
Let $\phi' : \pi_1(X) \rightarrow \C^{\ast}$ be the composition of $\phi$ with the natural projection
$\pi_1(X) \rightarrow \HH_1(X;\Z)$.  Let $\tilde{X}$ be the cover of $X$ corresponding to $\Ker(\phi')$.
Since the image of $\phi$ is finite, the cover $\tilde{X} \rightarrow X$ is finite.  The generalized
Riemann existence theorem (see \cite{GrauertRemmert} or \cite[\S XII]{SGA1}) implies that $\tilde{X}$
can be given the structure of a quasiprojective variety over $\C$.  Let $\tilde{E}_{\phi}$ be the pullback
of $E_{\phi}$ to $\tilde{X}$.  By construction, $\tilde{E}_{\phi}$ is the trivial flat line bundle
$\tilde{X} \times \C$.  This is clearly algebraic, so \cite[Proposition 2]{MumfordAbelian}
implies that $E_{\phi}$ is as well.
\end{proof}

\subsection{Orbifolds and their Picard groups}
\label{section:orbifolds}

\subsubsection{Topological theory}
\label{section:orbifoldstop}

We will not bother to set up a proper category of orbifolds, but simply give
definitions and theorems adapted to the needs of this paper (see \cite[Chapter 13]{ThurstonNotes}
for a more general discussion).  We begin with the following definitions.  

\begin{definition}
An {\em orbifold} consists of a pair $(X,G)$, where $X$ is a simply-connected CW complex and $G$ is a discrete
group acting properly discontinuously on $X$ by homeomorphisms.  To avoid technicalities,
we make the following assumptions.
\begin{itemize}
\item For all finite-index subgroups $G'$ of $G$, we require all the homology groups
of $X/G'$ to be finitely generated.
\item We require that $G$ contains some finite-index subgroup that acts freely on $X$.
\end{itemize}
An orbifold $(X,G)$ is a {\em trivial} orbifold if $G$ acts freely.  
If $(X_1,G_1)$ and $(X_2,G_2)$ are orbifolds, then
a map $(X_1,G_1) \rightarrow (X_2,G_2)$ consists of a map $f : X_1 \rightarrow X_2$
and a homomorphism $\phi : G_1 \rightarrow G_2$ such that $f(g \cdot x) = \phi(g) \cdot f(x)$
for all $x \in X_1$ and $g \in G_1$.  We will say that a map $(X_1,G_1) \rightarrow (X_2,G_2)$ is a {\em cover}
if $X_1 = X_2$, the map $f : X_1 \rightarrow X_2$ is the identity, and the homomorphism 
$\phi : G_1 \rightarrow G_2$ is injective.  The cover is {\em finite} if $G_1$ is a finite
index subgroup of $G_2$.
\end{definition}

\begin{remark}
The standard definition of an orbifold involves a local definition.  Our orbifolds are what
are usually called {\em good orbifolds}.
\end{remark}

\begin{remark}  
If $(X,G)$ is a trivial orbifold, then $G$ acts as a group of covering transformations on $X$.  In this case,
the notion of a cover reduces to the standard notion of a cover of the quotient $X/G$.
\end{remark}

We now define the topological Picard group of an orbifold.

\begin{definition}
Let $(X,G)$ be an orbifold.  A {\em $G$-equivariant complex line bundle} on $X$ is a $1$-dimensional complex
vector 
bundle $\rho : E \rightarrow X$ together with an action of $G$ on $E$ such that for all $g \in G$, the diagram
$$\begin{CD}
E @>{\tilde{x} \mapsto g \cdot \tilde{x}}>> E \\
@VV{\rho}V @VV{\rho}V \\
X @>{x \mapsto g \cdot x}>> X
\end{CD}$$
commutes and the induced maps $\rho^{-1}(x) \rightarrow \rho^{-1}(g \cdot x)$ are complex linear
for all $x \in X$.  The {\em topological Picard group} $\PicTop(X,G)$ of $(X,G)$
is the set of $G$-equivariant complex line bundles on $X$ (up to the
obvious equivalence).  The set $\PicTop(X,G)$ forms an abelian group under tensor products.
\end{definition}

\begin{remark}
If $(X,G)$ is a trivial orbifold, then $\PicTop(X,G) \cong \PicTop(X/G)$.
\end{remark}

\begin{remark}
A map $(X_1,G_1) \rightarrow (X_2,G_2)$ induces a homomorphism
$\PicTop(X_2,G_2) \rightarrow \PicTop(X_1,G_1)$.
\end{remark}

The following lemma is the extension of Lemma \ref{lemma:firstchern1} to general orbifolds.

\begin{theorem}
\label{theorem:firstchern2}
Let $(X,G)$ be an orbifold.  We then have an isomorphism $c_1 : \PicTop(X,G) \rightarrow \HH^2_G(X;\Z)$.
\end{theorem}

Here is one way to extract this theorem from the literature.  Let $H \lhd G$
be a finite-index normal subgroup that acts freely on $X$.  Define $G' = G/H$ and $X' = X/H$.  The
space $X'$ is not simply-connected and thus strictly speaking $(X',G')$ is not an orbifold,
but we will relax the condition of simple-connectivity in this proof.
Since $\PicTop(X,H) \cong \PicTop(X')$,
we have $\PicTop(X,G) \cong \PicTop(X',G')$.  Since $G'$ is a finite group, a theorem
of Lashof, May, and Segal \cite{LashofMaySegal} says that
$\PicTop(X',G') \cong \PicTop(EG' \times X',G')$, where $EG'$ is the universal
cover of a classifying space for $G'$.  The group $G'$ acts freely on $EG' \times X'$, 
so $\PicTop(EG' \times X',G') \cong \PicTop((EG' \times X')/(G'))$.  We can now apply
Lemma \ref{lemma:firstchern1} to get that
$$\PicTop((EG' \times X')/(G')) \cong \HH^2((EG' \times X')/(G');\Z) \cong \HH^2_{G'}(X';\Z).$$
Finally, the third conclusion of Lemma \ref{lemma:omnibuseqco} says that 
$\HH^2_{G'}(X';\Z) \cong \HH^2_G(X;\Z)$.

\begin{remark}
The theorem of Lashof, May, and Segal quoted above requires the structure groups of our
bundles to be compact; however, using the standard Gram-Schmidt process we can
reduce the structure groups of our bundles to $\text{U}(1) \cong S^1$.
\end{remark}

\begin{remark}
The earliest appearance of Theorem \ref{theorem:firstchern2} of which we are aware is 
the paper \cite{Conner} of Conner; however, he uses a rather different definition
of equivariant cohomology groups and it takes some work to show that his definition
is equivalent to ours.
\end{remark}

We close this section by discussing flat line bundles on orbifolds.  Their definition
is motivated by Lemma \ref{lemma:alternateflat}.

\begin{definition}
Let $(X,G)$ be an orbifold and let $\phi : \HH_1(G;\Z) \rightarrow \C^{\ast}$ be a homomorphism.  Let
$\phi' : G \rightarrow \C^{\ast}$ be the composition of $\phi$ with the natural projection
$G \rightarrow \HH_1(G;\Z)$.  The {\em flat line bundle} on $(X,G)$ defined by $\phi$ is 
the trivial complex line bundle $X \times \C$ with the $G$-action defined by the formula
$$g \cdot (x,z) = (g \cdot x, \phi'(g) z) \quad \quad (x \in \tilde{X}, z \in \C, g \in G).$$
\end{definition}

\noindent
The first two conclusions of Lemma \ref{lemma:omnibuseqco} imply that if $(X,G)$ is
an orbifold, then the torsion subgroup of $\HH^2_G(X;\Z)$ is isomorphic to
$\Hom((\HH_1(G;\Z))^{\Tor},\C^{\ast})$.  The following lemma is an immediate consequence
of this, Lemma \ref{lemma:alternateflat}, Lemma \ref{lemma:torsionpicard}, and the discussion
following Theorem \ref{theorem:firstchern2}.

\begin{lemma}
\label{lemma:torsionpicardorbifold}
Let $(X,G)$ be an orbifold.  The torsion subgroup of $\Pic(X,G)$ is the set
of elements of $\Pic(X,G)$ that can be endowed with the structure of a flat line bundle.
\end{lemma}

\subsubsection{Quasiprojective orbifolds}

We now discuss quasiprojective orbifolds.  For more information, see \cite{HainModuliTran}.

\begin{definition}
Let $(X,G)$ be an orbifold.  We will say that $(X,G)$ is a {\em quasiprojective orbifold} 
if it satisfies the following conditions.
\begin{itemize}
\item $X$ is a complex manifold and $G$ acts by biholomorphisms.
\item There is a finite-index subgroup $G' < G$ that acts freely on $X$ such that the complex manifold
$X/G'$ is a quasiprojective variety.  We will say that $(X,G')$ is a {\em quasiprojective finite cover} of $(X,G)$.
\end{itemize}
\end{definition}

\begin{remark}
If $(X,G)$ is a quasiprojective orbifold with a quasiprojective finite cover $(X,G')$ and $G''<G'$ is a finite-index subgroup, 
then the generalized Riemann existence theorem (see
\cite{GrauertRemmert} or \cite[\S XII]{SGA1}) implies that $X/G''$ is a quasiprojective variety, so $(X,G'')$ is also a quasiprojective
finite cover.
\end{remark}

\begin{remark}
If $(X,G)$ is a quasiprojective orbifold with a quasiprojective finite cover $(X,G')$ and if $(X,H)$ is a 
finite cover of $(X,G)$, then we can apply the previous remark to the subgroup $G'' = G' \cap H$ to deduce
that $(X,H)$ is a quasiprojective orbifold.
\end{remark}

We now discuss Picard groups of quasiprojective orbifolds.  In \S \ref{section:ordinarypicard} and
\S \ref{section:orbifoldstop}, which
discussed topological Picard groups, we assumed that our spaces were homeomorphic to CW complexes
with all their homology groups finitely generated.  A theorem of Lojasiewicz \cite{LojasiewiczTri}
says that this holds for quasiprojective varieties over $\C$, so the results of those sections
apply here.

\begin{definition}
If $(X,G)$ is a quasiprojective orbifold with a quasiprojective finite cover $(X,G')$, then
an {\em algebraic line bundle} on $(X,G)$ is a $G$-equivariant
holomorphic line bundle on $X$ that induces an algebraic line bundle on the quasiprojective
variety $X/G'$.  It is easy to see that this notion is independent of the choice of quasiprojective
finite cover.  The set of algebraic line bundles on $(X,G)$ forms the {\em Picard group} of $(X,G)$, which
we will denote by $\Pic(X,G)$.  The set $\Pic(X,G)$ forms an abelian group under tensor products.
We will denote by $\PicTriv(X,G)$ the kernel of the natural homomorphism $\Pic(X,G) \rightarrow \PicTop(X,G)$.
\end{definition}

We have the following lemma, which is an immediate corollary of Lemmas \ref{lemma:alternateflat}
and \ref{lemma:flatalgebraic}.

\begin{lemma}
\label{lemma:flatalgebraicorbifold}
If $(X,G)$ is a quasiprojective orbifold and $\mathcal{L}$ is a flat
line bundle on $(X,G)$, then $\mathcal{L}$ is algebraic.
\end{lemma}

\noindent
As a corollary of Lemmas \ref{lemma:torsionpicardorbifold} and \ref{lemma:flatalgebraicorbifold}, we
have the following.  Let $c_1 : \Pic(X,G) \rightarrow \HH^2_G(X;\Z)$ be the composition of the 
maps $\Pic(X,G) \rightarrow \PicTop(X,G)$ and $\PicTop(X,G) \rightarrow \HH^2_G(X;\Z)$.

\begin{lemma}
\label{lemma:torsionalgebraic}
If $(X,G)$ is a quasiprojective orbifold, then $c_1(\Pic(X,G)) \subset \HH^2_G(X;\Z)$ contains
the entire torsion subgroup.
\end{lemma}

If $X$ is a projective variety with $\HH^1(X;\Z) = 0$, then a standard result
says that $\PicTriv(X) = 0$, so we have an injection $c_1 : \Pic(X) \rightarrow \HH^2(X;\Z)$.
The following is the extension of this to quasiprojective orbifolds.

\begin{theorem}[{\cite[Theorem 14.3]{HainModuliTran}}]
\label{theorem:pictriv0}
Let $(X,G)$ be a quasiprojective orbifold with $\HH^1_G(X;\Z) = 0$.  Then
$\PicTriv(X,G) = 0$.
\end{theorem}

\subsubsection{Summary lemma}

We now give the following lemma, which summarizes the manner in which we will use
the above results about the Picard groups.  To make sense of its statement, recall that the first conclusion of 
Lemma \ref{lemma:omnibuseqco} says that if $(X,G)$ is an orbifold and $X$ is contractible, then 
$\HH^k_G(X;\Z) \cong \HH^k(G;\Z)$.

\begin{lemma}
\label{lemma:maintool}
Let $(X,G)$ be a quasiprojective orbifold with $X$ contractible and let $G' < G$ be a finite-index subgroup.
Assume that the following hold.
\begin{itemize}
\item $\HH_1(G;\Z) = 0$ and $\HH^1(G';\Z) = 0$.
\item $\HH_2(G;\Q) \cong \Z$ and $\HH_2(G';\Q) \cong \Q$.  
\item There exists $\lambda \in \Pic(X,G)$ such that $c_1(\lambda) \in \HH^2(G;\Z)$ generates $\HH^2(G;\Z)$ (we 
remark that the universal coefficients theorem implies that $\HH^2(G;\Z) \cong \Z$).
\item There exists some $n \geq 1$ and $\delta \in \Pic(X,G')$ such that the pullback 
of $\lambda$ to $\Pic(X,G')$ equals $n \delta + \tau$, where $\tau \in \Pic(X,G')$ is torsion.
\item The image of $\HH_2(G';\Z)$ in $\HH_2(G;\Z) \cong \Z$ is $m \Z$ for some $0 \leq m \leq n$.
\end{itemize}
Then the following hold.
\begin{itemize}
\item $\Pic(X,G') \cong \HH^2(G';\Z)$, so since $\HH_1(G';\Z)$ is torsion we have a short exact sequence
$$0 \longrightarrow \Hom(\HH_1(G';\Z),\C^{\ast}) \longrightarrow \Pic(X,G') \longrightarrow \Z \longrightarrow 0$$
coming from the universal coefficients theorem for $\HH^2(G';\Z)$.
\item $\Pic(X,G')$ is generated modulo torsion by $\delta$.
\item The image of $\HH_2(G';\Z)$ in $\HH_2(G;\Z) \cong \Z$ is $n \Z$.
\end{itemize}
\end{lemma}
\begin{proof}
The first step is to translate everything into statements about cohomology.
Since $\HH^1(G';\Z) = 0$, Theorems \ref{theorem:pictriv0} and \ref{theorem:firstchern2}
imply that we have an injection $c_1 : \Pic(X,G') \hookrightarrow \HH^2(G';\Z)$.  Let 
$\overline{\delta} = c_1(\delta)$.  We will prove the following two statements in the next paragraph.
\begin{enumerate}
\item The abelian group $\HH^2(G';\Z)$ is generated modulo torsion by $\overline{\delta}$.  
\item The image of $\HH_2(G';\Z)$ in $\HH_2(G;\Z) \cong \Z$ is $n \Z$.
\end{enumerate}
Lemma \ref{lemma:torsionalgebraic} says that
$\Pic(X,G') \subset \HH^2(G';\Z)$ contains the entire torsion subgroup of $\HH^2(G';\Z)$.  Since
$\delta$ is algebraic, we will be able to conclude that $\Pic(X,G') = \HH^2(G';\Z)$ 
and that $\Pic(X,G')$ is generated modulo torsion by $\delta$.   

It remains to prove the above two assertions.  For an abelian group $A$, recall that $A^{\Tf}$ denotes
the quotient of $A$ by its torsion subgroup $A^{\Tor}$.  
Let $\overline{\delta}' \in (\HH^2(G';\Z))^{\Tf}$ be the image of 
$\overline{\delta} \in \HH^2(G';\Z)$.  One of the assertions we want to prove is that
$\overline{\delta}'$ generates $(\HH^2(G';\Z))^{\Tf}$.  By the universal coefficients theorem, we have
$$(\HH^2(G;\Z))^{\Tf} = \Hom(\HH_2(G;\Z),\Z) \quad \quad \text{and} \quad \quad (\HH^2(G';\Z))^{\Tf} = \Hom(\HH_2(G';\Z),\Z).$$
Since $\HH_2(G;\Z) \cong \Z$ and $\HH_2(G';\Q) \cong \Q$, we deduce that both $\Hom(\HH_2(G;\Z),\Z)$ and
$\Hom(\HH_2(G';\Z),\Z)$ are isomorphic to $\Z$.  By assumption, the image of
$\HH_2(G';\Z))$ in $\HH_2(G;\Z) \cong \Z$ is $m \Z$ for some $0 \leq m \leq n$.  This implies that
the image of $\Hom(\HH_2(G;\Z),\Z)$ in $\Hom(\HH_2(G';\Z),\Z) \cong \Z$ is $m \Z$.  However, the image of
$\lambda \in \Pic(X,G)$ in $(\HH^2(G;\Z))^{\Tf} = \Hom(\HH_2(G;\Z),\Z) \cong \Z$ is a generator which
maps by assumption to $n \overline{\delta}'$.  We conclude that $m = n$ (assertion 2 above) and that
$\overline{\delta}'$ generates $(\HH^2(G';\Z))^{\Tf}$ (assertion 1 above).
\end{proof}

\section{The symplectic group}
\label{section:symplecticgroup}

The key to proving our theorems about $\PPAV_g(L)$ is the exact sequence
$$1 \longrightarrow \Sp_{2g}(\Z,L) \longrightarrow \Sp_{2g}(\Z) \longrightarrow \Sp_{2g}(\Z/L) \longrightarrow 1.$$
To analyze it, we will need some results about $\Sp_{2g}(\Z,L)$ and $\Sp_{2g}(\Z/L)$ that 
are collected in \S \ref{section:splabel}--\ref{section:decompthm}.

\begin{remark}
The fact that the map $\Sp_{2g}(\Z) \rightarrow \Sp_{2g}(\Z/L)$ is surjective is nontrivial.  For a proof,
see \cite[Theorem 1]{NewmanSmart}.
\end{remark}

\subsection{The abelianization of $\Sp_{2g}(\Z,L)$}
\label{section:splabel}

The starting point here is the following theorem of Newman and Smart.

\begin{theorem}[{Newman-Smart, \cite[Theorem 7]{NewmanSmart}}]
\label{theorem:splquotient}
Consider $g \geq 2$ and $L,L' \geq 2$ such that $L' \mid L$.  We then
have an $\Sp_{2g}(\Z)$-equivariant isomorphism
$\Sp_{2g}(\Z,L) / \Sp_{2g}(\Z,L L') \cong \SpLie_{2g}(\Z/L')$.
\end{theorem}

\begin{remark}
Newman and Smart do not mention the fact that this isomorphism is $\Sp_{2g}(\Z)$-equivariant,
but this fact is evident from their proof.
\end{remark}

The isomorphism in Theorem \ref{theorem:splquotient} is constructed as follows.  For simplicity,
we will deal with the case that $L' = L \geq 2$, so we will construct a map
$\phi : \Sp_{2g}(\Z,L) \rightarrow \SpLie_{2g}(\Z/L)$.
An element $M \in \Sp_{2g}(\Z,L)$ can be written in the form $\One_{2g} + L \cdot A$ for some matrix $A$.  Define
$\phi(M)$ to equal $A$ modulo $L$.  Newman and Smart prove that $\phi(M) \in \SpLie_{2g}(\Z/L)$ and that
every element of $\SpLie_{2g}(\Z/L)$ can be so obtained.
It is easy see that this is a homomorphism whose kernel
is $\Sp_{2g}(\Z,L^2)$.  Summing up, we have a short exact sequence
\begin{equation}
\label{eqn:splsplie}
1 \longrightarrow \Sp_{2g}(\Z,L^2) \longrightarrow \Sp_{2g}(\Z,L) \longrightarrow \SpLie_{2g}(\Z/L) \longrightarrow 1.
\end{equation}

For $L$ odd and $g \geq 3$, the following result (which was proven independently by Perron, Sato,
and the author) shows that $\phi$ gives the entire abelianization of $\Sp_{2g}(\Z,L)$.

\begin{theorem}[{Perron \cite{Perron}, Putman \cite{PutmanAbel}, Sato \cite{Sato}}]
\label{theorem:splabel}
Fix $g \geq 3$ and $L > 2$ such that $L$ is odd.  We then have $[\Sp_{2g}(\Z,L),\Sp_{2g}(\Z,L)] = \Sp_{2g}(\Z,L^2)$.
\end{theorem}

\begin{corollary}
\label{corollary:splabel}
Fix $g \geq 3$ and $L > 2$ such that $L$ is odd.  We then have an $\Sp_{2g}(\Z)$-equivariant isomorphism
$\HH_1(\Sp_{2g}(\Z,L);\Z) \cong \SpLie_{2g}(\Z/L)$.
\end{corollary}

\begin{remark}
A similar result for congruence subgroups of $\SL_n(\Z)$ was proved by Lee and Szczarba \cite{LeeSzczarba}.
In that case, the restriction to $L$ odd is not necessary.
\end{remark}

Now assume that $L$ is even.  In this case, the abelianization of $\Sp_{2g}(\Z,L)$ is slightly larger
than $\SpLie_{2g}(\Z/L)$.  Define $\Sp_{2g}(\Z,L,2L)$ to be the set of block matrices 
$\MatTwoTwo{A}{B}{C}{D}$ in $\Sp_{2g}(\Z,L)$ such that the diagonal entries of $B$ and
$C$ are equal to $0$ modulo $2L$.  Igusa proved the following.

\begin{theorem}[{Igusa, \cite[Lemma 1]{IgusaGraded}}]
\label{theorem:igusa}
The following hold for $g \geq 1$ and $L$ even.
\begin{itemize}
\item $\Sp_{2g}(\Z,L,2L)$ is a normal subgroup of $\Sp_{2g}(\Z)$.
\item $\Sp_{2g}(\Z,L) / \Sp_{2g}(\Z,L,2L) \cong \HH_1(\Sigma_g;\Z/2)$ as $\Sp_{2g}(\Z)$-modules.
\item The group $\Sp_{2g}(\Z,L) / \Sp_{2g}(\Z,L^2,2L^2)$ is abelian.
\end{itemize}
\end{theorem}

\begin{remark}
The second and third conclusions of Theorem \ref{theorem:igusa} do not appear in quite this
form in \cite{IgusaGraded}.
For the second, Igusa states that 
$\Sp_{2g}(\Z,L) / \Sp_{2g}(\Z,L,2L) \cong (\Z/2)^{2g}$ via an isomorphism that takes
$\One_g + L \MatTwoTwo{A}{B}{C}{D} \in \Sp_{2g}(\Z,L)$ to
the $2g$-dimensional vector whose entries are the diagonals of $B$ and $C$ modulo $2$.  The
fact that this is an $\Sp_{2g}(\Z)$-equivariant isomorphism to $\HH_1(\Sigma_g;\Z/2)$ is
implicit in his proof.
For the third, Igusa only claims that the group $\Sp_{2g}(\Z,L) / \Sp_{2g}(\Z,2L,4L)$
is abelian; however, his proof actually gives the stronger conclusion that
$\Sp_{2g}(\Z,L) / \Sp_{2g}(\Z,L^2,2L^2)$ is abelian.
\end{remark}

Applying Theorem \ref{theorem:igusa}, we can take the quotient of exact sequence \eqref{eqn:splsplie} by 
$\Sp_{2g}(\Z,L^2,2L^2)$ and get an exact sequence
$$0 \longrightarrow \HH_1(\Sigma_g;\Z/2) \longrightarrow \Sp_{2g}(\Z,L) / \Sp_{2g}(\Z,L^2,2L^2) \longrightarrow
\SpLie_{2g}(\Z/L) \longrightarrow 0.$$
of abelian groups with actions of $\Sp_{2g}(\Z)$.  Sato proved that we have 
obtained the abelianization of $\Sp_{2g}(\Z,L)$.

\begin{theorem}[{Sato \cite{Sato}}]
\label{theorem:sato}
Fix $g \geq 3$ and $L \geq 2$ such that $L$ is even.  We then have 
$$[\Sp_{2g}(\Z,L),\Sp_{2g}(\Z,L)] = \Sp_{2g}(\Z,L^2,2L^2).$$
\end{theorem}

\begin{corollary}
\label{corollary:sato}
Fix $g \geq 3$ and $L \geq 2$ such that $L$ is even.  We then have a short exact sequence
$$0 \longrightarrow \HH_1(\Sigma_g;\Z/2) \longrightarrow \HH_1(\Sp_{2g}(\Z,L);\Z) \longrightarrow \SpLie_{2g}(\Z/L)
\longrightarrow 0$$
of $\Sp_{2g}(\Z)$-modules.
\end{corollary}

\subsection{Four results about $\Sp_{2g}(\Z/L)$}
\label{section:decompthm}

We will need four results about $\Sp_{2g}(\Z/L)$, the first two of which are the following group cohomology calculations.

\begin{lemma}
\label{lemma:splperfect}
For $g \geq 3$ and $L \geq 2$, we have $\HH_1(\Sp_{2g}(\Z/L);\Z) = 0$.
\end{lemma}
\begin{proof}
The group $\Sp_{2g}(\Z/L)$ is the surjective image of $\Sp_{2g}(\Z)$, which satisfies
$\HH_1(\Sp_{2g}(\Z);\Z)=0$ (see Theorem \ref{theorem:hsp} below).  The lemma follows.
\end{proof}

\begin{theorem}[{M.\ Stein, \cite[Theorem 2.13 and Proposition 3.3.a]{Stein2}}]
\label{theorem:stein}
For $g \geq 3$ and $L \geq 2$ such that $4 \nmid L$, we have $\HH_2(\Sp_{2g}(\Z/L);\Z) = 0$.
\end{theorem}

\begin{remark}
For $L$ prime, Theorem \ref{theorem:stein} was originally proven by Steinberg \cite{SteinbergSchur}.  The calculations
in \cite{SteinbergSchur} and \cite{Stein2}, which deal with arbitrary Chevalley groups, are closely related to
Milnor's definition of $K_2$ of a ring.  In the general context of Chevalley groups, the condition $4 \nmid L$ is
necessary; indeed, for large $n$ we have $\HH_2(\SL_n(\Z/L);\Z)$ equal to $0$ if $4 \nmid L$ and equal
to $\Z/2$ if
$4 \mid L$.  This latter result is due to Dennis; see \cite[Corollary 10.8]{MilnorKTheory} and \cite[\S
12]{DennisSteinSurvey}.
\end{remark}

\begin{remark}
The results of \cite{Stein2} imply that to remove the condition $4 \nmid L$ from Theorem \ref{theorem:stein}, it would
be enough to prove that $\HH_2(\Sp_{6}(\Z/4);\Z) = 0$.  In theory this could be proven with a computer; however, the group
$\Sp_{6}(\Z/4)$ is quite large and complicated (it has cardinality {\tt 4106059776000}), making this computer calculation
impractical.
\end{remark}

The other two lemmas we need allow us to express $\Sp_{2g}(\Z/L)$ in terms of $\Sp_{2g}(\Z/p)$
for $p$ prime.  The first is the following, which is due to Newman and Smart.

\begin{lemma}[{Newman-Smart, \cite[Theorem 5]{NewmanSmart}}]
\label{lemma:spldecomp}
Fix $g \geq 1$ and $L \geq 2$.  Write $L = p_1^{k_1} \cdots p_n^{k_n}$,
where the $p_i$ are distinct primes.  We then have
$$\Sp_{2g}(\Z/L) \cong \bigoplus_{i=1}^n \Sp_{2g}(\Z/{p_i^{k_i}}).$$
\end{lemma}

The second is as follows.

\begin{lemma}
\label{lemma:sppdecomp}
Fix $g \geq 3$, a prime $p$, and some $k \geq 1$.  We then have a short exact
sequence
\begin{equation}
\label{eqn:sppdecomp}
1 \longrightarrow \SpLie_{2g}(\Z/{p}) \longrightarrow \Sp_{2g}(\Z/{p^{k+1}}) \longrightarrow \Sp_{2g}(\Z/{p^{k}}) \longrightarrow 1.
\end{equation}
Moreover, if $p^k \geq 4$, then this exact sequence does not split.
\end{lemma}

\begin{remark}
The condition $p^k \geq 4$ is not necessary, but the indicated result is
sufficient for our purposes and has an easier proof.
\end{remark}

\begin{proof}[Proof of Lemma \ref{lemma:sppdecomp}]
By Theorem \ref{theorem:splquotient}, we can obtain \eqref{eqn:sppdecomp} by quotienting
the exact sequence
$$1 \longrightarrow \Sp_{2g}(\Z,p^k) \longrightarrow \Sp_{2g}(\Z) \longrightarrow \Sp_{2g}(\Z/{p^k}) \longrightarrow 1,$$
by $\Sp_{2g}(\Z,p^{k+1})$.  We must only prove that \eqref{eqn:sppdecomp} does not split if $p^k \geq 4$.

Assume that \eqref{eqn:sppdecomp} does split, and let $\varphi : \Sp_{2g}(\Z/{p^k}) \rightarrow \Sp_{2g}(\Z/{p^{k+1}})$ be
a splitting.  For an integer $L$, we will during this proof denote by $\One_{L}$ and $\Zero_{L}$ 
the $2g \times 2g$ identity and zero matrices with entries in $\Z/L$, respectively (in the rest of this paper, the subscripts
refer to the dimensions of the matrices).  Also, we will
denote by $E_{L}$ the $2g \times 2g$ matrix with entries in $\Z/L$ with a $1$ at position
$(1,g+1)$ and $0$'s elsewhere.  Observe that $\One_{p^k} + E_{p^k} \in \Sp_{2g}(\Z/{p^k})$ and that
$(\One_{p^k} + E_{p^k})^{p^k} = 1$.  There exists a matrix $A$ with entries in $\Z/{p^{k+1}}$
such that $\varphi(\One_{p^k} + E_{p^k}) = \One_{p^{k+1}} + E_{p^{k+1}} + p^{k} A$.  Observe that
\begin{align*}
\One_{p^{k+1}} &= (\One_{p^{k+1}} + E_{p^{k+1}} + p^{k} A)^{p^k} \\
&= \One_{p^{k+1}} + p^{k}(E_{p^{k+1}} + p^{k} A) + \binom{p^{k}}{2} (E_{p^{k+1}} + p^{k} A)^2 + \cdots + (E_{p^{k+1}} + p^{k} A)^{p^k}.
\end{align*}
Since $p^{k}(E_{p^{k+1}} + p^{k} A) = p^k E_{p^{k+1}} \neq 0$, we will obtain a contradiction
if we show that all the higher order terms in this sum vanish.  

Consider $\binom{p^k}{n}(E_{p^{k+1}} + p^{k} A)^n$ for some $2 \leq n \leq p^k$.  We can expand
out $(E_{p^{k+1}} + p^{k} A)^n$ as a sum of terms each of which for some $0 \leq m \leq n$ 
is a product of $m$ factors of $E_{p^{k+1}}$ and
$(n-m)$ factors of $p^{k} A$, taken in some order.  If one of these terms contains two factors of $p^{k} A$, then
it must vanish.  Also, since $E_{p^{k+1}}^2 = 0$, if one of these terms contains two adjacent factors of
$E_{p^{k+1}}$, then it must vanish.  If $n \geq 4$, then we conclude that $(E_{p^{k+1}} + p^{k} A)^n = 0$.  If
$2 \leq n \leq 3$, then we might not have $(E_{p^{k+1}} + p^{k} A)^n = 0$, but each nonzero term
in it must contain a factor of $p^k A$, so we deduce that $(E_{p^{k+1}} + p^{k} A)^n$ is a multiple of $p^k$.  Since
$p^k \geq 4$, it follows that $2 \leq n < p^k$.  We then have
that $\binom{p^k}{n}$ is divisible by $p$, so $\binom{p^k}{n}(E_{p^{k+1}} + p^{k} A)^n$ is a multiple of $p^{k+1}$; i.e.\
it vanishes.
\end{proof}

\section{The twisted first homology groups of $\Sp_{2g}(\Z/L)$}
\label{section:twistedcoefficients}

This section has three subsections.  In \S \ref{section:hcoho}, we show that 
$\HH_1(\Sp_{2g}(\Z/L);\HH_1(\Sigma_g;\Z/2)) \cong \Z/2$ for $g \geq 2$ and $L$ even with $4 \nmid L$, 
slightly generalizing a result of Pollatsek \cite{Pollatsek}.  
In \S \ref{section:adjointcohoproof}, we prove Theorem \ref{theorem:adjointcoho}, which asserts
that $\HH_1(\Sp_{2g}(\Z/L);\SpLie_{2g}(\Z/L)) = 0$ for $g \geq 3$ and $L \geq 2$ such that $4 \nmid L$.  This
is preceded by \S \ref{section:adjointcohoprelim}, where a number of preliminary results are collected.

\subsection{The first homology group of $\Sp_{2g}(\Z/L)$ with coefficients $\HH_1(\Sigma_g;\Z/2)$}
\label{section:hcoho}

Assume that $G$ is a group and $M$ is a $G$-vector space over a field of characteristic not equal to $2$.  Furthermore,
assume that there is some element in the center of $G$ that acts by $-1$ on $M$.  It is an easy exercise in elementary
group cohomology (whose details we leave to the reader since this result is only needed for motivation) 
that under these conditions we have $\HH^1(G;M) = 0$.
In particular, if $p$ is an odd prime, then this applies to the action of $\Sp_{2g}(\Z/p)$ on $\HH_1(\Sigma_g;\Z/p)$,
since we have $-\One_{2g}$ in the center of $\Sp_{2g}(\Z/p)$.

For $p = 2$, this fails, and in fact Pollatsek \cite{Pollatsek} showed that 
$\HH^1(\Sp_{2g}(\Z/2);\HH_1(\Sigma_g;\Z/2)) \cong \Z/2$ for $g \geq 2$.  The following theorem slightly extends this.

\begin{theorem}
\label{theorem:hcoho}
For $g \geq 3$ and $L \geq 2$ such that $2 \mid L$ and $4 \nmid L$, 
we have $\HH_1(\Sp_{2g}(\Z/L);\HH_1(\Sigma_g;\Z/2)) \cong \Z/2$.
\end{theorem}

\begin{remark}
Computer calculations indicate that the condition $4 \nmid L$ is probably unnecessary.
\end{remark}

\begin{proof}[{Proof of Theorem \ref{theorem:hcoho}}]
The algebraic intersection form is a nondegenerate pairing on $\HH_1(\Sigma_g;\Z/2)$, so we have an isomorphism
$\HH_1(\Sigma_g;\Z/2) \cong (\HH_1(\Sigma_g;\Z/2))^{\ast}$
of $\Sp_{2g}(\Z/L)$-modules.  Using Lemma \ref{lemma:duality},
this implies that it is enough to prove that $\HH^1(\Sp_{2g}(\Z/L);\HH_1(\Sigma_g;\Z/2)) \cong \Z/2$.  Set $L'=L/2$.  By
Lemma \ref{lemma:spldecomp}, we have $\Sp_{2g}(\Z/L) \cong \Sp_{2g}(\Z/2) \oplus \Sp_{2g}(\Z/{L'})$ (this is
where we use the fact that $4 \nmid L$).  The action of $\Sp_{2g}(\Z/L)$ on $\HH_1(\Sigma_g;\Z/2)$
factors through $\Sp_{2g}(\Z/2)$.  The subgroup $\Sp_{2g}(\Z/{L'})$ of $\Sp_{2g}(\Z/L)$ is
the kernel of the projection $\Sp_{2g}(\Z/L) \rightarrow \Sp_{2g}(\Z/2)$, so $\Sp_{2g}(\Z/{L'})$ 
acts trivially on $\HH_1(\Sigma_g;\Z/2)$.  Lemma \ref{lemma:splperfect} says that 
$\HH_1(\Sp_{2g}(\Z/{L'});\Z) = 0$.  We can thus apply Lemma \ref{lemma:killchunk}
to deduce that
$\HH^1(\Sp_{2g}(\Z/L);\HH_1(\Sigma_g;\Z/2)) \cong \HH^1(\Sp_{2g}(\Z/2);\HH_1(\Sigma_g;\Z/2))$.
Pollatsek's theorem \cite{Pollatsek} says that this is $\Z/2$, and the result follows.
\end{proof}

\subsection{Some preliminaries for Theorem \ref{theorem:adjointcoho}}
\label{section:adjointcohoprelim}

\subsubsection{The transgression in the cohomology Hochschild-Serre spectral sequence}
In \cite{Huebschmann}, Huebschmann gives a recipe for computing the first
transgression in the Hochschild-Serre spectral sequence.  We will need a very special
case of this theorem.

Consider a short exact sequence of groups
\begin{equation}
\label{eqn:huebshort}
1 \longrightarrow A \longrightarrow G \longrightarrow Q \longrightarrow 1,
\end{equation}
where $A$ is abelian.  Regarding $A$ as a $G$-module, there is an associated Hochschild-Serre
spectral sequence converging to $\HH^{\ast}(G;A)$.  The first transgression
of this spectral sequence is the differential
$$d : \HH^0(Q;\HH^1(A;A)) \longrightarrow \HH^2(Q;\HH^0(A;A)).$$
Now, $\HH^0(Q;\HH^1(A;A)) \cong \Hom_Q(A,A)$
contains a distinguished element, namely, the identity homomorphism $i : A \rightarrow A$.  There
is thus a canonical element $d(i) \in \HH^2(Q;\HH^0(A;A)) \cong \HH^2(Q;A)$ associated to the extension \eqref{eqn:huebshort}.
The Euler class of \eqref{eqn:huebshort} provides another canonical element of $\HH^2(Q;A)$.
Huebschmann proved that these two elements are equal.

\begin{theorem}[{Huebschmann, \cite[Theorem 6]{Huebschmann}}]
\label{theorem:transgression}
Up to signs, $d(i) \in \HH^2(Q;A)$ is the Euler class of the extension \eqref{eqn:huebshort}.  In
particular, if \eqref{eqn:huebshort} is a nontrivial extension, then $d(i)$ is nonzero.
\end{theorem}

\begin{remark}
In fact, Huebschmann proves a much more general result than Theorem \ref{theorem:transgression},
but in our special case it is easy to see that Huebschmann's theorem reduces to Theorem \ref{theorem:transgression}.
\end{remark}

\subsubsection{The symplectic Lie algebra}
\label{section:symplecticlie}

We will need two facts about $\SpLie_{2g}(\Z/p)$.  They are both analogues of familiar
results in characteristic $0$, but their proofs are somewhat different.
The first is the following.  

\begin{proposition}
\label{proposition:spliehomo}
For $g \geq 2$ and $p$ an odd prime, the group
$\Hom_{\Sp_{2g}(\Z/p)}(\SpLie_{2g}(\Z/p),\SpLie_{2g}(\Z/p))$ is cyclic of order
$p$ and generated by the identity.
\end{proposition}

For the proof of Proposition \ref{proposition:spliehomo}, we will need the following theorem.
In it, we denote by $\ClosedField_p$ the algebraic closure of the field with $p$ elements.

\begin{theorem}
\label{theorem:splieirred}
For $g \geq 1$ and $p$ an odd prime, $\SpLie_{2g}(\ClosedField_p)$ is an irreducible $\Sp_{2g}(\Z/p)$-module.
\end{theorem}

Theorem \ref{theorem:splieirred} can be extracted from the literature as follows.
A theorem of Jacobson \cite[Lemma 7]{Jacobson}
says that $\SpLie_{2g}(\ClosedField_p)$ is an irreducible Lie algebra (we remark that Jacobson
showed that this fails if $p=2$).  The standard theory of algebraic groups
then says that $\SpLie_{2g}(\ClosedField_p)$ is an irreducible representation of
$\Sp_{2g}(\ClosedField_p)$ (see \cite[\S 6]{BorelChevalley} for more details
on this point).  A theorem of Curtis \cite{CurtisProjective} finally 
allows us to conclude that this representation remains irreducible upon being
restricted to $\Sp_{2g}(\Z/p) < \Sp_{2g}(\ClosedField_p)$, as desired.

\begin{remark}
Curtis's paper \cite{CurtisProjective} makes the assumption that $p > 7$, but
this assumption is not necessary; see
\cite[Corollary 7.3]{BorelChevalley} for the details of the general case.
\end{remark}

\begin{proof}[{Proof of Proposition \ref{proposition:spliehomo}}]
Consider
$$f \in \Hom_{\Sp_{2g}(\Z/{p})}(\SpLie_{2g}(\Z/p),\SpLie_{2g}(\Z/p)).$$
We get an induced $\Sp_{2g}(\Z/{p})$-module homomorphism
$\overline{f} : \SpLie_{2g}(\ClosedField_p) \rightarrow \SpLie_{2g}(\ClosedField_p)$.  
Theorem \ref{theorem:splieirred} says that $\SpLie_{2g}(\ClosedField_p)$ is an 
irreducible $\Sp_{2g}(\Z/p)$-module, so Schur's lemma implies that $\overline{f}$ is
scalar multiplication by an element of $\ClosedField_p$.  The homomorphism $f$ is then multiplication by
the same element, so we deduce that in fact $f$ is scalar multiplication by an element of $\Z/p$, and
the proposition follows.
\end{proof}

The second result we will need is as follows.
\begin{lemma}
\label{lemma:selfdual}
For $g \geq 1$ and $p$ an odd prime, the $\Sp_{2g}(\Z/p)$-modules $\SpLie_{2g}(\Z/p)$ and $(\SpLie_{2g}(\Z/p))^{\ast}$
are isomorphic.
\end{lemma}
\begin{remark}
Lemma \ref{lemma:selfdual} is false for $p = 2$.  Indeed, if it were true, then by Lemma \ref{lemma:duality}
we would have 
$$\HH_1(\Sp_{2g}(\Z/2);\SpLie_{2g}(\Z/2)) \cong \HH^1(\Sp_{2g}(\Z/2);\SpLie_{2g}(\Z/2)).$$  
However,
while we will prove below that $\HH_1(\Sp_{2g}(\Z/2);\SpLie_{2g}(\Z/2))=0$, a computer calculation similar
to the one we will describe in Step 1 of the proof of Theorem \ref{theorem:adjointcoho} shows that
$\HH^1(\Sp_{2g}(\Z/2);\SpLie_{2g}(\Z/2)) \cong \Z/2$.
\end{remark}

If we were working in characteristic $0$, then Lemma \ref{lemma:selfdual} would follow from the nondegeneracy of
the Killing form $\Trace(\adj(x) \adj(y))$.
However, it turns out that the Killing form is degenerate if $p \mid g+1$ (on \cite[p.\ 47]{Seligman}, this
observation is attributed to Dynkin \cite{Dynkin}).  Thankfully, in our situation something
even more simple-minded works.  Namely, elements of $\SpLie_{2g}(\Z/p)$ are {\em matrices}, so
we can define
$$(x,y) = \Trace(xy).$$
This is clearly a symmetric bilinear form invariant under the action of $\Sp_{2g}(\Z/p)$, and
the following lemma says that it is nondegenerate.
\begin{lemma}
\label{lemma:traceform}
If $g \geq 1$ and $p$ is an odd prime, then $(\cdot,\cdot)$ is a nondegenerate bilinear form on $\SpLie_{2g}(\Z/p)$.
\end{lemma}
\noindent
This lemma, which has Lemma \ref{lemma:selfdual} as an immediate corollary, 
seems to be well-known.  For instance, it is asserted without proof on \cite[p.\ 47]{Seligman}.
However, we have been unable to find a reference proving it, so we sketch a proof.
\begin{proof}[Proof of Lemma \ref{lemma:traceform}]
We will write down a basis $S$ for $\SpLie_{2g}(\Z/p)$ such that for $x \in S$, there is a unique
$x' \in S$ with $\Trace(x x') \neq 0$.  Choosing an ordering of the elements of $S$, it
will follow that the matrix $M$ representing the bilinear form $(\cdot,\cdot)$ has a unique
non-zero element in each row.  Since the form is symmetric, $M$ will also have a unique non-zero
element in each column.  We will thus thus be able to conclude that up to signs, the 
determinant of $M$ is the product of these non-zero elements, and the lemma will follow.

For $1 \leq i,j \leq 2g$, let $E_{i,j}$ denote the matrix with a $1$ at position $(i,j)$ and $0$'s elsewhere.
Define
$$S = \{\text{$A_{i,j}$ $|$ $1 \leq i,j \leq g$}\} \cup \{\text{$B_i,B_i'$ $|$ $1 \leq i \leq g$}\} \cup \{\text{$C_{i,j},C_{i,j}'$ $|$ $1 \leq i < j \leq g$}\} \subset \SpLie_{2g}(\Z/p),$$
where
\begin{align*}
A_{i,j} = E_{i,j} - E_{g+i,g+j}, &\quad C_{i,j} = E_{g+i,j} + E_{g+j,i}, \quad C_{i,j}' = E_{i,g+j} + E_{j,g+i},\\
&B_{i} = E_{g+i,i}, \quad B_{i}' = E_{i,g+i}.
\end{align*}
It is clear that $S$ is a basis for $\SpLie_{2g}(\Z/p)$.  The lemma is now a simple case by case check.
For instance, the only $x \in S$ with $\Trace(A_{i,j} x) \neq 0$ is
$A_{j,i}$, and
$$\Trace(A_{i,j} A_{j,i}) = \Trace((E_{i,j} - E_{g+i,g+j})(E_{j,i} - E_{g+j,g+i})) = \Trace(E_{i,i} + E_{g+j,g+j}) = 2.$$
The other cases are similar.
\end{proof}

\subsubsection{Homological stability}
Fix a ring $R$.  Define a {\em coefficient system} for
$\Sp_{2g}(R)$ to be a sequence $\{M_g,i_g\}_{g \geq 1}$, where $M_g$ is an $\Sp_{2g}(R)$-module
and $i_g : M_g \rightarrow M_{g+1}$ is an map that is equivariant with respect to
the natural inclusion $\Sp_{2g}(R) \hookrightarrow \Sp_{2(g+1)}(R)$.  One example of such a coefficient
system is $M_g = \SpLie_{2g}(R)$, where the maps $i_g : M_g \rightarrow M_{g+1}$ are the obvious
inclusions.

For $R$ satisfying a certain connectivity condition (see below), 
Charney \cite{CharneyVogtmann} proved a homological stability result for
$\Sp_{2g}(R)$ with respect to coefficient systems that are {\em central of degree $k$}
for some $k \geq 0$.  We will not reproduce the definition
of this here, but merely record the following fact, whose proof is similar to that of Example
3 on \cite[p.\ 123]{CharneyVogtmann}.

\begin{lemma}
\label{lemma:liecentral}
The coefficient system $\SpLie_{2g}(R)$ is central of degree $2$.
\end{lemma}

\noindent
The statement of Charney's theorem depends on the connectivity of a certain space $HU(R^{2g})$, the
{\em poset of hyperbolic unimodular sequences} in $R^{2g}$ (equipped with the standard symplectic form).  The
key property of $HU(R^{2g})$ for us is the following lemma.

\begin{lemma}
\label{lemma:huconnected}
For $g \geq 1$ and $L \geq 2$, the complex $HU((\Z/L)^{2g})$ is $\frac{g-4}{2}$-connected.
\end{lemma}

Lemma \ref{lemma:huconnected} can be extracted from the literature in the following way.  In
\cite[Theorem 5.8]{MirzaiiVanDerKallen}, Mirzaii and van der Kallen 
proved a much more general result, namely, that if $R$ has {\em stable rank}
$\StableRank(R)$, then $HU(R^{2g})$ is $\frac{g-\StableRank(R)-3}{2}$-connected.  The lemma
then follows from the fact that $\Z/L$ has stable rank $1$ (see \cite[Proposition V.3.4.a]{BassKTheory}; a
ring $R$ has stable rank $n$ in the sense of \cite{MirzaiiVanDerKallen} if it satisfies
Bass's condition $SR_{n+1}(R)$).  

We can now state Charney's theorem.

\begin{theorem}[{Charney, \cite[Theorem 4.3]{CharneyVogtmann}}]
\label{theorem:charneystability}
Let $R$ be a finitely generated ring and let $a \geq 3$ be such that $HU(R^{2g})$ is $\frac{g-a-1}{2}$-connected for
all $g \geq 1$.  Let $\{M_g,i_g\}_{g \geq 1}$ be a coefficient system for $\Sp_{2g}(R)$ that 
is central of degree $k$.  Then the natural map
$$\HH_i(\Sp_{2g}(R);M_g) \rightarrow \HH_{i}(\Sp_{2(g+1)}(R);M_{g+1})$$
is surjective for $g \geq 2i+k+a$ and bijective for $g \geq 2i+k+a+1$.
\end{theorem}

\begin{remark}
Charney proved a result similar to Lemma \ref{lemma:huconnected} for $R$ a PID.  This is enough to cover
most of our uses of Theorem \ref{theorem:charneystability}; however, in Step 2 of the proof
of Theorem \ref{theorem:adjointcoho} below, we will need to apply Theorem \ref{theorem:charneystability}
with $R = \Z/9$, which is not a domain.
\end{remark}

\subsection{The proof of Theorem \ref{theorem:adjointcoho}}
\label{section:adjointcohoproof}

Recall that we wish to prove that
$$\HH_1(\Sp_{2g}(\Z/L);\SpLie_{2g}(\Z/L)) = 0$$
for $g \geq 3$ and $L \geq 2$ such that $4 \nmid L$.  The proof will have four steps.  In all the steps, we fix $g \geq 3$.

\BeginSteps
\begin{step}
Let $p$ be prime.  Then $\HH_1(\Sp_{2g}(\Z/{p});\SpLie_{2g}(\Z/p))=0$.  
\end{step}

Lemma \ref{lemma:duality} says that
\begin{equation}
\label{eqn:duality}
(\HH_1(\Sp_{2g}(\Z/p);\SpLie_{2g}(\Z/p)))^{\ast} \cong \HH^1(\Sp_{2g}(\Z/p);(\SpLie_{2g}(\Z/p))^{\ast}).
\end{equation}
It is enough to prove that either side of \eqref{eqn:duality} vanishes.  
For $p$ odd, Lemma \ref{lemma:selfdual} says that $(\SpLie_{2g}(\Z/p))^{\ast} \cong \SpLie_{2g}(\Z/p)$.
A theorem of Volklein \cite{Volklein} says that $\HH^1(\Sp_{2g}(\Z/p);\SpLie_{2g}(\Z/p)) = 0$ for $g \geq 3$ and primes $p > 5$.
We thus only need to deal with $p \in \{2,3,5\}$.  By the homological stability theorem for the
symplectic group (Theorem \ref{theorem:charneystability}) together with 
Lemmas \ref{lemma:liecentral} and \ref{lemma:huconnected}, 
for a fixed $p$ the desired result will hold for all $g$ if it holds for 
$g$ such that $3 \leq g \leq 7$.  This reduces the result to checking a finite number of cases.
These cases may be easily checked on a computer using the algorithm for calculating twisted first cohomology groups given
in \cite[\S 7.6.1]{Holt}.  Computer code for performing these calculations
(written in {\tt C++}) can be found on the author's webpage.

\begin{step}
Let $p$ be an odd prime and let $k \geq 1$.  Then
$\HH_1(\Sp_{2g}(\Z/{p^k});\SpLie_{2g}(\Z/p)) = 0$.
\end{step}

Since $p$ is odd, Lemma \ref{lemma:duality} together with Lemma \ref{lemma:selfdual} says that the desired
claim is equivalent to proving that $\HH^1(\Sp_{2g}(\Z/{p^k});\SpLie_{2g}(\Z/p)) = 0$.
The proof of this will be by induction on $k$.  The case $k=1$ follows from Step 1.  Our inductive argument also
does not cover the case $(p,k)=(3,2)$, but this case can be handled by a computer calculation just
as in Step 1.

Assume now that $k \geq 1$ and that $\HH^1(\Sp_{2g}(\Z/{p^k});\SpLie_{2g}(\Z/p)) = 0$.  If $p = 3$, then
we will also assume that $k \geq 2$.  We must prove that $\HH^1(\Sp_{2g}(\Z/{p^{k+1}});\SpLie_{2g}(\Z/p)) = 0$.
By Lemma \ref{lemma:sppdecomp} (which requires $p^k \geq 4$), we have a nonsplit short exact sequence 
\begin{equation}
\label{eqn:induction1}
1 \longrightarrow \SpLie_{2g}(\Z/{p}) \longrightarrow \Sp_{2g}(\Z/{p^{k+1}}) \longrightarrow \Sp_{2g}(\Z/{p^{k}}) \longrightarrow 1.
\end{equation}
The inclusion $\SpLie_{2g}(\Z/{p}) \hookrightarrow \Sp_{2g}(\Z/p^{k+1})$ in this exact sequence
takes a matrix $A \in \SpLie_{2g}(\Z/{p})$ to $1+p^k A \in \Sp_{2g}(\Z/p^{k+1})$.
Observe that the restriction of the action of $\Sp_{2g}(\Z/{p^{k+1}})$ on $\SpLie_{2g}(\Z/{p})$ to the subgroup
$\SpLie_{2g}(\Z/p) < \Sp_{2g}(\Z/p^{k+1})$ indicated in \eqref{eqn:induction1} is trivial.
The $E_2$ page of the associated Hochschild-Serre spectral sequence for cohomology with coefficients
in $\SpLie_{2g}(\Z/p)$ is of the form
\begin{center}
\begin{tabular}{|c@{\hspace{0.2 in}}c@{\hspace{0.2 in}}c}
\footnotesize{$\HH^0(\Sp_{2g}(\Z/{p^{k}});\HH^1(\SpLie_{2g}(\Z/p);\SpLie_{2g}(\Z/p)))$} & & \\
\footnotesize{$\ast$} & \footnotesize{$\HH^1(\Sp_{2g}(\Z/{p});\SpLie_{2g}(\Z/p))$} & \footnotesize{$\HH^2(\Sp_{2g}(\Z/{p^k});\SpLie_{2g}(\Z/p))$} \\
\cline{1-3}
\end{tabular}
\end{center}
By assumption, we have have $\HH^1(\Sp_{2g}(\Z/{p^k});\SpLie_{2g}(\Z/p))=0$.  Let
$$d : \HH^0(\Sp_{2g}(\Z/{p^{k}});\HH^1(\SpLie_{2g}(\Z/p);\SpLie_{2g}(\Z/p))) \longrightarrow \HH^2(\Sp_{2g}(\Z/{p^k});\SpLie_{2g}(\Z/p))$$
be the differential.  Since the action of
$\Sp_{2g}(\Z/{p^k})$ on $\SpLie_{2g}(\Z/p)$ factors through $\Sp_{2g}(\Z/p)$, we can use
Proposition \ref{proposition:spliehomo} to deduce that
$$\HH^0(\Sp_{2g}(\Z/{p^{k}});\HH^1(\SpLie_{2g}(\Z/p);\SpLie_{2g}(\Z/p))) \cong \Hom_{\Sp_{2g}(\Z/{p^k})}(\SpLie_{2g}(\Z/p),\SpLie_{2g}(\Z/p)).$$
is an order $p$ cyclic group generated by the identity homomorphism $i : \SpLie_{2g}(\Z/p) \rightarrow \SpLie_{2g}(\Z/p)$.
Since \eqref{eqn:induction1} does not split, Theorem \ref{theorem:transgression} implies that
$d(i) \neq 0$.  We conclude that the $E_3$ page of the spectral sequence is of the form
\begin{center}
\begin{tabular}{|c@{\hspace{0.2 in}}c@{\hspace{0.2 in}}c}
$0$ & & \\
$\ast$ & $0$ & $\ast$ \\
\cline{1-3}
\end{tabular}
\end{center}
and the desired result follows.

\begin{step}
Let $L \geq 2$ satisfy $4 \nmid L$ and let $p$ be a prime dividing $L$.  Then $\HH_1(\Sp_{2g}(\Z/{L});\SpLie_{2g}(\Z/p)) = 0$.
\end{step}

Write $L = p_1^{k_1} \cdots p_n^{k_n}$, where the $p_i$ are distinct primes and $p_1 = p$.  By
Lemma \ref{lemma:spldecomp}, we have 
$$\Sp_{2g}(\Z/L) \cong \bigoplus_{i=1}^n \Sp_{2g}(\Z/{p_i^{k_i}}).$$
For $i > 1$, the action of $\Sp_{2g}(\Z/{p_i^{k_i}})$ on $\SpLie_{2g}(\Z/p)$ is trivial.  Also, 
Lemma \ref{lemma:splperfect}
says that $\HH_1(\Sp_{2g}(\Z/{p_i^{k_i}});\Z)=0$ for all $i$.  Lemma \ref{lemma:killchunk} thus implies
that
$$\HH_1(\Sp_{2g}(\Z/L);\SpLie_{2g}(\Z/p)) \cong \HH_1(\Sp_{2g}(\Z/{{p_1}^{k_1}});\SpLie_{2g}(\Z/p)).$$
If $p_1$ is odd, then Step 2 says that $\HH_1(\Sp_{2g}(\Z/{p_1^{k_1}});\SpLie_{2g}(\Z/p)) = 0$.  If $p_1 = 2$, then
the fact that $4 \nmid L$ implies that $k_1 = 1$, so Step 1 says that $\HH_1(\Sp_{2g}(\Z/{p_1^{k_1}});\SpLie_{2g}(\Z/p)) = 0$.
The result follows.

\begin{step}
Let $L \geq 2$ satisfy $4 \nmid L$ and let 
$L' > 1$ satisfy $L' \mid L$.  Then $\HH_1(\Sp_{2g}(\Z/{L});\SpLie_{2g}(\Z/{L'})) = 0$.  In
particular, $\HH_1(\Sp_{2g}(\Z/{L});\SpLie_{2g}(\Z/{L})) = 0$.
\end{step}

Write $L = p_1^{k_1} \cdots p_n^{k_n}$ and $L' = p_1^{k_1'} \cdots p_n^{k_n'}$, where the $p_i$ are 
distinct odd primes and $k_i \geq 1$ for $1 \leq i \leq n$.  Observe that $0 \leq k_i' \leq k_i$ for $1 \leq i \leq n$.  
The proof will be by induction on $\sum |k_i'|$.

The base case is $\sum |k_i'| = 1$, in which case $L' = p_j$ for some $1 \leq j \leq n$ and the result
follows from Step 3. Assume now that $\sum |k_i'| > 1$ and that
the desired result is true for all smaller sums.  Pick $1 \leq m \leq n$ such that $k_m' \geq 1$.  Set
$$k_i'' = \begin{cases}
k_i' & \text{if $i \neq m$} \\
k_i'-1 & \text{if $i = m$}
\end{cases} \quad \quad (1 \leq i \leq n)$$
and $L'' = p_1^{k_1''} \cdots p_n^{k_n''}$.  There is a short exact sequence
$$0 \longrightarrow \Z/{L''} \longrightarrow \Z/{L'} \longrightarrow \Z/{p_m} \longrightarrow 0$$
of rings and a corresponding short exact sequence
$$0 \longrightarrow \SpLie_{2g}(\Z/{L''}) \longrightarrow \SpLie_{2g}(\Z/{L'}) \longrightarrow \SpLie_{2g}(\Z/{p_m}) \longrightarrow 0$$
of $\Sp_{2g}(\Z/L)$-modules.  This induces a long exact sequence in $\Sp_{2g}(\Z/L)$-homology which contains the
segment
$$\HH_1(\Sp_{2g}(\Z/{L});\SpLie_{2g}(\Z/{L''})) \rightarrow \HH_1(\Sp_{2g}(\Z/{L});\SpLie_{2g}(\Z/{L'})) \rightarrow \HH_1(\Sp_{2g}(\Z/{L});\SpLie_{2g}(\Z/{p_m})).$$
The inductive hypothesis implies that 
$$\HH_1(\Sp_{2g}(\Z/{L});\SpLie_{2g}(\Z/{L''}))=0 \quad \text{and} \quad \HH_1(\Sp_{2g}(\Z/{L});\SpLie_{2g}(\Z/{p_m}))=0,$$
so $\HH_1(\Sp_{2g}(\Z/{L});\SpLie_{2g}(\Z/{L'}))=0$, as desired.

\section{The moduli space of principally polarized abelian varieties}
\label{section:ppav}

\subsection{Definitions}
\label{section:ppavdef}

We now introduce the moduli space of principally polarized abelian varieties.  For more information,
see \cite{FreitagSingular, IgusaTheta, MumfordTata1}.  In this paper, all abelian varieties are over $\C$.

The {\em Siegel upper half plane} $\Siegel_g$ is the space of $g \times g$ symmetric complex
matrices with positive definite imaginary parts.  A principal polarization of an abelian variety $A$ is
an extra bit of data whose precise nature does not concern us here.  All we will need to know is
that it provides a symplectic form on $\HH_1(A;\Z)$.  The space $\Siegel_g$ parametrizes pairs $(A,B)$, where
$A$ is a $g$-dimensional principally polarized abelian variety and $B$ is a symplectic basis for $\HH_1(A;\Z)$.  The
abelian variety associated to $\Omega \in \Siegel_g$ is $\C^g / (\Z^g + \Omega \Z^g)$.  The space $\Siegel_g$ is
a contractible complex manifold and the group $\Sp_{2g}(\Z)$ acts on it holomorphically and properly discontinuously.  
If $M \in \Sp_{2g}(\Z)$ and $\Omega \in \Siegel_g$, then
$$M(\Omega) = \frac{A \Omega + B}{C \Omega + D},$$
where $M = \MatTwoTwo{A}{B}{C}{D}$ is the decomposition of $M$ into $g \times g$ blocks.  One can view
this action as changing the symplectic basis associated to an element of $\Siegel_g$. 

As a complex analytic space, $\PPAV_g = \Siegel_g / \Sp_{2g}(\Z)$; however, for us it is important
to take into account its structure as an orbifold.
To simplify our statements, we will denote $\Sp_{2g}(\Z)$ and $\PPAV_g$ by $\Sp_{2g}(\Z,1)$ and $\PPAV_g(1)$, respectively.  
For $L \geq 1$, the space $\PPAV_g(L)$ is the orbifold
$(\Siegel_g,\Sp_{2g}(\Z,L))$.  Thus $\PPAV_g(L')$ is a covering orbifold for $\PPAV_g(L)$ whenever $L|L'$.  For
$L \geq 3$, the group $\Sp_{2g}(\Z,L)$ acts freely on $\Siegel_g$, so in these cases $\PPAV_g(L)$ is a
trivial orbifold.  Ignoring the orbifold structure for a moment, it is known that the complex analytic space
$\Siegel_g / \Sp_{2g}(\Z,L)$ is a quasiprojective variety for all $L$ (see \cite{IgusaTheta}
for a proof and history of this result).  We conclude that $\PPAV_g(L)$ is a quasiprojective orbifold
with quasiprojective finite cover $\PPAV_g(L')$ for any $L' \geq 3$ such that $L|L'$.

\subsection{Some classical results}
\label{section:ppavclassical}

To investigate $\Pic(\PPAV_g(L))$, our main tool will be Lemma \ref{lemma:maintool}.
To verify the conditions of that lemma, we will need a number of classical results, the first of which
are the following group-cohomology calculations. 

\begin{theorem}
\label{theorem:hsp}
For $g \geq 3$, we have $\HH_1(\Sp_{2g}(\Z);\Z) = 0$.  For $g \geq 4$, we have $\HH_2(\Sp_{2g}(\Z);\Z) \cong \Z$.
\end{theorem}

\begin{remark}
Theorem \ref{theorem:hsp} is classical, but I do not know a good reference for it.  See the remark
after the proof of Lemma \ref{lemma:modsp} below for one way to extract it from the literature.
\end{remark}

\begin{theorem}[{Kazhdan, \cite{KazhdanT}}]
\label{theorem:kazhdan}
Fix $g \geq 2$ and $L \geq 1$.  We then have $\HH^1(\Sp_{2g}(\Z,L);\Z) = 0$.
\end{theorem}

\begin{theorem}[{Borel, \cite{BorelStability1,BorelStability2}}]
\label{theorem:borel}
Fix $g \geq 3$ and $L \geq 1$.  We then have $\HH_2(\Sp_{2g}(\Z,L);\Q) \cong \Q$.
\end{theorem}
\begin{remark}
The range of $g$ in Theorem \ref{theorem:borel} is better than that claimed
in \cite{BorelStability1}, and follows from a criterion that can be found in \cite{BorelStability2}.
\end{remark}

Recall that $\Pic(\PPAV_g)$ is the set of $\Sp_{2g}(\Z)$-equivariant line bundles on $\Siegel_g$.  
There is a standard line bundle $\PPAVLine_g \in \Pic(\PPAV_g)$ 
whose sections are the Siegel modular forms (of weight $1$ and level $1$; see \cite{FreitagSingular} or \cite{MumfordTata1}).
The fiber of $\PPAVLine_g$ over a point of $\Siegel_g$ can be identified with the space of holomorphic $1$-forms on
the associated abelian variety.  The following theorem gives its first Chern class.  This theorem is folklore and
I am not sure to whom to attribute it.  One proof of it can be found in \cite[Theorem 17.4]{HainModuliTran}.

\begin{theorem}
\label{theorem:ppavlinegen}
For $g \geq 3$, the class $c_1(\PPAVLine_g)$ generates $\HH^2(\Sp_{2g}(\Z);\Z) \cong \Z$.
\end{theorem}

Define $\PPAVLine_g(L) \in \Pic(\PPAV_g(L))$ to be the pullback of $\PPAVLine_g$.  
For even levels, the transformation formulas for the theta-nulls (see, e.g.,  
\cite[\S I.5]{FreitagSingular}) show that they provide a square root for $\PPAVLine_g(L)$ when $L$ is even.  We
thus have the following divisibility result.

\begin{lemma}
\label{lemma:thetanulls}
For $g \geq 3$ and $L \geq 2$ such that $L$ is even, $\PPAVLine_g(L) \in \Pic(\PPAV_g(L))$ is divisible by $2$.
\end{lemma}

\subsection{The proofs of Theorems \ref{theorem:ppav} and \ref{theorem:ppavzgen} and \ref{theorem:splh2}}
\label{section:proofsppav}

We wish to apply Lemma \ref{lemma:maintool} with $(X,G) = (\Siegel_g,\Sp_{2g}(\Z))$ and $G' = \Sp_{2g}(\Z,L)$
and $\lambda = \PPAVLine_g$ and
$$n = \begin{cases}
1 & \text{if $L$ is odd,} \\
2 & \text{if $L$ is even.}
\end{cases}$$
The results of \S \ref{section:ppavdef}--\ref{section:ppavclassical} 
show that all the conditions of Lemma \ref{lemma:maintool} are satisfied except possibly
the condition that the image $A$ of $\HH_2(\Sp_{2g}(\Z,L);\Z)$ in $\HH_2(\Sp_{2g}(\Z);\Z) \cong \Z$ is
$m \Z$ for some $m$ with $1 \leq m \leq n$, which we now prove (we remark that Lemma \ref{lemma:maintool} will
then show that $m=n$).

We will use the Hochschild-Serre spectral sequence in group homology arising from the short exact sequence
$$1 \longrightarrow \Sp_{2g}(\Z,L) \longrightarrow \Sp_{2g}(\Z) \longrightarrow \Sp_{2g}(\Z/L) \longrightarrow 1.$$
Theorem \ref{theorem:stein} says that $\HH_2(\Sp_{2g}(\Z/L);\Z) = 0$.
Theorem \ref{theorem:edgegroups} thus implies that the $E^{\infty}$ page of our spectral sequence is of the form
\begin{center}
\begin{tabular}{|c@{\hspace{0.2 in}}c@{\hspace{0.2 in}}c}
$A$ & &  \\
$\ast$ & $B$ & \\
$\ast$ & $\ast$ & $0$ \\
\cline{1-3}
\end{tabular}
\end{center}
with $B$ a quotient of $\HH_1(\Sp_{2g}(\Z/L);\HH_1(\Sp_{2g}(\Z,L);\Z))$.  We thus have a short exact sequence
$$0 \longrightarrow A \longrightarrow \HH_2(\Sp_{2g}(\Z);\Z) \longrightarrow B \longrightarrow 0.$$
Recall that our goal is to show that $A = \HH_2(\Sp_{2g}(\Z);\Z)$ (i.e.\ that $B=0$) if $L$ is odd and that
$A$ has index at most $2$ in $\HH_2(\Sp_{2g}(\Z);\Z)$ (i.e.\ that $B$ is either $0$ or $\Z/2$) if $L$ is even 
(by the above it will then follow that in fact $B \cong \Z/2$).  Since $B$ is a quotient of $\HH_1(\Sp_{2g}(\Z/L);\HH_1(\Sp_{2g}(\Z,L);\Z))$, it is enough
to show that $\HH_1(\Sp_{2g}(\Z/L);\HH_1(\Sp_{2g}(\Z,L);\Z))$ is $0$ if $L$ is odd and is either $0$ or $\Z/2$ if $L$
is even.

If $L$ is odd, then Corollary \ref{corollary:splabel} says that $\HH_1(\Sp_{2g}(\Z,L);\Z) \cong \SpLie_{2g}(\Z/L)$,
so we can apply Theorem \ref{theorem:adjointcoho} to deduce that
$$\HH_1(\Sp_{2g}(\Z/L);\HH_1(\Sp_{2g}(\Z,L);\Z)) \cong \HH_1(\Sp_{2g}(\Z/L);\SpLie_{2g}(\Z/L)) = 0.$$
If $L$ is even, then Corollary \ref{corollary:sato} gives a short exact sequence
$$0 \longrightarrow \HH_1(\Sigma_g;\Z/2) \longrightarrow \HH_1(\Sp_{2g}(\Z,L);\Z) \longrightarrow \SpLie_{2g}(\Z/L)
\longrightarrow 0$$
of $\Sp_{2g}(\Z/L)$-modules.  The long exact sequence in $\Sp_{2g}(\Z/L)$-homology contains the exact sequence
$$\HH_1(\Sp_{2g}(\Z/L);\HH_1(\Sigma_g;\Z/2)) \longrightarrow \HH_1(\Sp_{2g}(\Z/L);\HH_1(\Sp_{2g}(\Z,L);\Z))
\longrightarrow \HH_1(\Sp_{2g}(\Z/L);\SpLie_{2g}(\Z/L)).$$
Using Theorems \ref{theorem:adjointcoho} and \ref{theorem:hcoho}, this reduces to
$$\Z/2 \longrightarrow \HH_1(\Sp_{2g}(\Z/L);\HH_1(\Sp_{2g}(\Z,L);\Z)) \longrightarrow 0,$$
so $\HH_1(\Sp_{2g}(\Z/L);\HH_1(\Sp_{2g}(\Z,L);\Z))$ is either $0$ or $\Z/2$, as desired.

\section{The Torelli group}
\label{section:torelli}

\subsection{Definition and generators}

To prove our theorems about the moduli space of curves, we will need some results about the Torelli group
(see \cite{JohnsonSurvey} for a survey).
Fix some $g \geq 1$ and $b \leq 1$.  Recall that the Torelli group $\Torelli_{g,b}$ is the subgroup of $\Mod_{g,b}$
consisting of mapping classes that act trivially on $\HH_1(\Sigma_{g,b};\Z)$.  
We thus have a short exact sequence
$$1 \longrightarrow \Torelli_{g,b} \longrightarrow \Mod_{g,b} \longrightarrow \Sp_{2g}(\Z) \longrightarrow 1.$$
For several proofs that the map $\Mod_{g,b} \rightarrow \Sp_{2g}(\Z)$ is surjective, see \cite[\S 7.1]{FarbMargalitBook}. 

Two key types of elements of $\Torelli_{g,b}$ are as follows.  Let $T_x$ denote the Dehn twist about a simple closed
curve $x$.
\begin{enumerate}
\item If $x$ is a non-nullhomotopic separating curve (see Figure \ref{figure:torelli}.a), 
then $x$ is nullhomologous, and it follows that $T_x \in \Torelli_{g,b}$.  We
will call $T_x$ a {\em separating twist}.
\item If $\{y,z\}$ is a pair of disjoint non-homotopic simple closed nonseparating curves such that $y \cup z$
separates $\Sigma_{g,b}$ (see Figure \ref{figure:torelli}.a), then $y$ and $z$ are homologous.  Thus $T_y$ and
$T_z$ act in the same way on $\HH_1(\Sigma_{g,b};\Z)$, so $T_y T_{z}^{-1} \in \Torelli_{g,b}$.  We will
call $\{y,z\}$ a {\em bounding pair} and $T_y T_z^{-1}$ a {\em bounding pair map}.
\end{enumerate}
Following work of Birman \cite{BirmanSiegel}, it was proven by Powell \cite{PowellTorelli} that $\Torelli_{g,b}$
is generated by the set of all bounding pair maps and separating twists.
For $g \geq 3$, this was improved by Johnson.

\begin{theorem}[{Johnson, \cite{JohnsonGen}}]
\label{theorem:torelligen}
For $g \geq 3$ and $b \leq 1$, the group $\Torelli_{g,b}$ is generated by bounding pair maps.
\end{theorem}

\begin{remark}
Later, Johnson \cite{JohnsonFinite} proved that $\Torelli_{g,b}$ is generated by finitely many bounding
pair maps for $g \geq 3$.  This should be contrasted with work of McCullough and Miller \cite{McCulloughMiller},
who showed that $\Torelli_2$ is {\em not} finitely generated.  It is still unknown whether or not
$\Torelli_{g,b}$ is finitely presentable for $g \geq 3$.
\end{remark}

\subsection{Abelianization}

\subsubsection{The Johnson homomorphism}
We now turn to the abelianization of $\Torelli_{g,b}$.  There are two key abelian quotients of $\Torelli_{g,b}$.  The
first is the {\em Johnson homomorphism}, which was constructed by Johnson in \cite{JohnsonHomo}.  The precise
definition of this homomorphism will not be needed in this paper (it comes from the action 
of $\Torelli_{g,b}$ on the second nilpotent truncation of $\pi_1(\Sigma_{g,b})$), so
we only list the properties of it that we need.
Letting $H = \HH_1(\Sigma_{g,b};\Z)$, it is a surjective $\Mod_{g,b}$-equivariant homomorphism of the form
$$\tau : \Torelli_{g,1} \longrightarrow \wedge^3 H \quad \text{or} \quad \tau : \Torelli_{g} \longrightarrow (\wedge^3 H)/H,$$
depending on whether $b$ is $1$ or $0$.  
Here $H$ is embedded in $\wedge^3 H$ as follows.  Let $\{a_1,b_1,\ldots,a_g,b_g\}$ be a symplectic basis for $H$.  The
element $\omega = a_1 \wedge b_1 + \cdots + a_g \wedge b_g \in \wedge^2 H$ is then independent of the basis, and we have
an injection $H \hookrightarrow \wedge^3 H$ that takes $h \in H$ to $h \wedge \omega$.  

\subsubsection{The Birman-Craggs-Johnson homomorphism}
The second abelian quotient of $\Torelli_{g,b}$ is given by the Birman-Craggs-Johnson homomorphism.  In
\cite{BirmanCraggs}, Birman and Craggs constructed a large family of $\Z/2$ quotients of $\Torelli_{g,b}$ using
the Rochlin invariant of homology 3-spheres.  Later,
in \cite{JohnsonBirmanCraggs}, Johnson packaged all these abelian quotients into one homomorphism.  Its target
is easiest to describe when $b = 1$.  Let $\Omega$ be the set of all $\Z/2$-quadratic forms on $\HH_1(\Sigma_{g,b};\Z/2)$
inducing the standard intersection form.  In other words, $\Omega$ is the set of all maps 
$f : \HH_1(\Sigma_{g,b};\Z/2) \rightarrow \Z/2$ satisfying $f(x+y)=f(x)+f(y)+i(x,y)$ for $x,y \in \HH_1(\Sigma_{g,b};\Z/2)$, 
where $i(\cdot,\cdot)$ is the algebraic intersection form.  The target of the Birman-Craggs-Johnson homomorphism
is a certain additive subspace of the ring $\Map(\Omega,\Z/2)$ of $\Z/2$-valued functions on $\Omega$.

For each $x \in \HH_1(\Sigma_{g,b};\Z/2)$, there is a function $\overline{x} : \Omega \rightarrow \Z/2$ that 
takes $f \in \Omega$ to $f(x)$.  It is clear that $\overline{x}^2 = \overline{x}$ and
$\overline{x+y} = \overline{x} + \overline{y} + i(x,y)$ for $x,y \in \HH_1(\Sigma_{g,b};\Z/2)$; here $i(x,y) \in \Z/2$
denotes the constant function.  
Motivated by this, we make the following definitions.
Let $B'(2g)$ denote the ring of polynomials with $\Z/2$ coefficients in the (formal) symbols 
$\{\text{$\overline{x}$ $|$ $x \in \HH_1(\Sigma_{g,b};\Z/2)$, $x \neq 0$}\}$.  Next, let $B(2g)$ denote the quotient of
$B'(2g)$ by the ideal generated by the set of relations
\begin{align*}
\overline{x}^2 = \overline{x} \quad & \text{for $x \in \HH_1(\Sigma_{g,b};\Z/2)$ with $x \neq 0$}, \\
\overline{x+y} = \overline{x} + \overline{y} + i(x,y) \quad & \text{for $x,y \in \HH_1(\Sigma_{g,b};\Z/2)$ with $x,y \neq 0$ and $x \neq y$}.
\end{align*}
One can view elements of $B(2g)$ as ``square-free'' $\Z/2$-polynomials in a fixed basis for $\HH_1(\Sigma_g;\Z/2)$.  As such,
element of $B(2g)$ have a natural degree.  Let $B_n(2g)$ denote the elements of $B(2g)$ of degree at most $n$.  The group
$B_n(2g)$ is an abelian $2$-group, and the group $\Sp_{2g}(\Z)$ acts on it.  For a nonzero $x \in \HH_1(\Sigma_{g,b};\Z/2)$
and $f \in \Sp_{2g}(\Z)$, this action takes $\overline{x}$ to $\overline{f(x)}$.  There is also a natural
$\Sp_{2g}(\Z)$-equivariant ring homomorphism $B(2g) \rightarrow \Map(\Omega,\Z/2)$.

Johnson proved that the map $B_n(2g) \rightarrow \Map(\Omega,\Z/2)$ is injective for $n \leq 3$, and the image
of $B_3(2g)$ in $\Map(\Omega,\Z/2)$ is the target
of the Birman-Craggs-Johnson homomorphism on a surface with one boundary component.  Summing up, the
Birman-Craggs-Johnson homomorphism on a surface with one boundary component is a surjective homomorphism
$$\sigma : \Torelli_{g,1} \longrightarrow B_3(2g)$$
that is $\Mod_{g,1}$-equivariant.

For closed surfaces, we have to take a further quotient of $B_3(2g)$.  Associated to every $f \in \Omega$ is its
{\em Arf invariant}, which can be computed as follows.  Let $\{a_1,b_1,\ldots,a_g,b_g\}$ be a symplectic basis
for $\HH_1(\Sigma_{g,b};\Z/2)$.  Then the Arf invariant of $f$ is $f(a_1) f(b_1) + \cdots + f(a_g) f(b_g)$.  One
can show that this is independent of the symplectic basis and that two elements of $\Omega$ are in the same
$\Sp_{2g}(\Z/2)$ orbit if and only if they have the same Arf invariant.  Let $\Omega_0 \subset \Omega$ be the
set of quadratic forms of Arf invariant $0$.  There is then a restriction map 
$\pi : \Map(\Omega,\Z/2) \rightarrow \Map(\Omega_0,\Z/2)$.  Let $\overline{B}_n(2g) = \pi(B_n(2g))$ for $n \leq 3$.  The
Birman-Craggs-Johnson homomorphism on a closed surface is a surjective $\Mod_g$-equivariant homomorphism
$$\sigma : \Torelli_{g} \longrightarrow \overline{B}_3(2g).$$
See \cite{JohnsonBirmanCraggs} for a calculation of the dimension of $\overline{B}_3(2g)$.

\subsubsection{Combining the abelian quotients}
As maps to abelian groups, both the Johnson and the Birman-Craggs-Johnson homomorphisms factor through $\HH_1(\Torelli_{g,b};\Z)$,
and we will frequently regard them as homomorphisms with domain $\HH_1(\Torelli_{g,b};\Z)$.
In \cite{JohnsonAbel}, Johnson proved that the Johnson homomorphisms and the Birman-Craggs-Johnson homomorphisms
combine to give the abelianization of $\Torelli_{g,b}$ when $g \geq 3$.  However, they are not independent.
Let $H = \HH_1(\Sigma_{g,b};\Z)$ and $H_2 = \HH_1(\Sigma_{g,b};\Z/2)$.  Observe that $B_3(2g) / B_2(2g) \cong \wedge^3 H_2$.  
Johnson proved the following theorem.

\begin{theorem}[{Johnson, \cite{JohnsonAbel}}]
\label{theorem:johnsonabel1}
For $g \geq 3$, there is a commutative diagram
\begin{equation}
\label{eqn:johnsonhomobcghomodiagram}
\begin{CD}
0 @>>> B_2(2g) @>>> \HH_1(\Torelli_{g,1};\Z) @>{\tau}>> \wedge^3 H @>>> 0\\
@.     @VV{=}V      @VV{\sigma}V                        @VVV \\
0 @>>> B_2(2g) @>>> B_3(2g)                  @>>>       \wedge^3 H_2 @>>> 0
\end{CD}
\end{equation}
In particular, $\HH_1(\Torelli_{g,1};\Z) \cong B_2(2g) \oplus \wedge^3 H$.
\end{theorem}

\begin{remark}
The top exact sequence in \eqref{eqn:johnsonhomobcghomodiagram} does {\em not} split in a manner compatible
with the action of $\Sp_{2g}(\Z)$ on $\HH_1(\Torelli_{g,b};\Z)$.  Thus the decomposition
$\HH_1(\Torelli_{g,1};\Z) \cong B_2(2g) \oplus \wedge^3 H$ is not an isomorphism of $\Sp_{2g}(\Z)$-modules.
\end{remark}

Similarly, one can show that $\overline{B}_3(2g) / \overline{B}_2(2g) \cong (\wedge^3 H_2) / H_2$, and Johnson
proved the following.

\begin{theorem}[{Johnson, \cite{JohnsonAbel}}]
For $g \geq 3$, there is a commutative diagram
\begin{equation}
\label{eqn:johnsonhomobcghomodiagram2}
\begin{CD}
0 @>>> \overline{B}_2(2g) @>>> \HH_1(\Torelli_{g};\Z) @>{\tau}>> (\wedge^3 H) / H @>>> 0\\
@.     @VV{=}V                 @VV{\sigma}V                      @VVV \\
0 @>>> \overline{B}_2(2g) @>>> \overline{B}_3(2g)     @>>>       (\wedge^3 H_2) / H_2 @>>> 0
\end{CD}
\end{equation}
In particular, $\HH_1(\Torelli_g;\Z) \cong \overline{B}_2(2g) \oplus (\wedge^3 H) / H$.
\end{theorem}

\subsection{Coinvariants}

Recall that if a group $G$ acts on an abelian group $A$, then the coinvariants of that action are denoted $A_G$.  The
goal of this section is to prove Proposition \ref{proposition:torellicoinvariants} below, which gives
$(\HH_1(\Torelli_{g,b};\Z))_{\Sp_{2g}(\Z,L)}$.  This is proceeded by two lemmas.  In the statement of the first one,
we will write $\Sp_{2g}(\Z,1)$ and $\Mod_{g}(1)$ for $\Sp_{2g}(\Z)$ and $\Mod_{g}$, respectively.  

\Figure{figure:torelli}{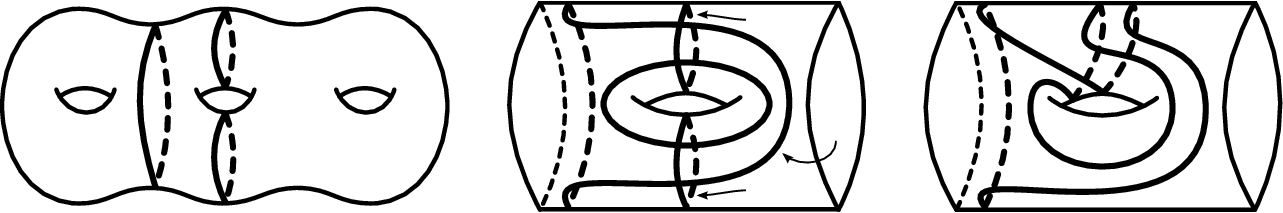}{a. A separating twist and a bounding pair map \CaptionSpace
b--c. The crossed lantern relation $(T_{y_1}T_{y_2}^{-1})(T_{x_1}T_{x_2}^{-1})=(T_{z_1}T_{z_2}^{-1})$.}

\begin{lemma}
\label{lemma:kill2torsion}
For $g \geq 3$ and $b \leq 1$ and $L \geq 1$, let $v \in (\HH_1(\Torelli_{g,b};\Z))_{\Sp_{2g}(\Z,L)}$.  
Then $L \cdot v = 0$.
\end{lemma}

\begin{remark}
In \cite[proof of Theorem 1.1]{McCarthyLevel}, McCarthy proved that for $v \in \HH_1(\Mod_{g,b}(L);\Z)$, we have
$L \cdot v = 0$.  Though this is related to Lemma \ref{lemma:kill2torsion}, it seems difficult to use
McCarthy's techniques to prove Lemma \ref{lemma:kill2torsion}.
\end{remark}

\begin{proof}[{Proof of Lemma \ref{lemma:kill2torsion}}]
For $f \in \Torelli_{g,b}$, we will denote the associated element of $(\HH_1(\Torelli_{g,b};\Z))_{\Sp_{2g}(\Z,L)}$
by $[f]_L$.  Theorem \ref{theorem:torelligen} says that $\Torelli_{g,b}$ is generated by bounding pair maps, so it
is enough to show that $L [T_{x_1} T_{x_2}^{-1}]_L = 0$ for a bounding pair map $T_{x_1} T_{x_2}^{-1}$.
Embed $\{x_1,x_2\}$ in a 2-holed torus as in Figure \ref{figure:torelli}.b.  We will make use
of the {\em crossed lantern relation} from \cite{PutmanInfinite}.  Letting
$\{y_1,y_2\}$ and $\{z_1,z_2\}$ be the other bounding pair maps depicted in Figures
\ref{figure:torelli}.b--c, this relation says that
\begin{equation}
\label{eqn:crossedlantern}
(T_{y_1}T_{y_2}^{-1})(T_{x_1}T_{x_2}^{-1})=(T_{z_1}T_{z_2}^{-1}).
\end{equation}
Observe that for $i=1,2$ we have $z_i = T_{x_2}(y_i)$.  The key observation is that for all $n \geq 0$, conjugating \eqref{eqn:crossedlantern}
by $T_{x_2}^n$ results in another crossed lantern relation
$$(T_{T_{x_2}^n(y_1)} T_{T_{x_2}^n(y_2)}^{-1}) (T_{x_1}T_{x_2}^{-1}) = (T_{T_{x_2}^{n+1}(y_1)} T_{T_{x_2}^{n+1}(y_2)}^{-1}).$$
Since $T_{x_2}^L \in \Mod_{g,n}(L)$, we conclude that in $(\HH_1(\Torelli_{g,n};\Z)_{\Sp_{2g}(\Z,L)}$ we have
$[T_{y_1} T_{y_2}^{-1}]_L$ equal to
\begin{align*}
[T_{x_2}^L (T_{y_1} T_{y_2}^{-1}) T_{x_2}^{-L}]_L &= [(T_{T_{x_2}^L(y_1)} T_{T_{x_2}^L(y_2)}^{-1})]_L\\
                       &= [T_{x_1}T_{x_2}^{-1}]_L + [(T_{T_{x_2}^{L-1}(y_1)} T_{T_{x_2}^{L-1}(y_2)}^{-1})]_L\\
                       &= 2[T_{x_1}T_{x_2}^{-1}]_L + [(T_{T_{x_2}^{L-2}(y_1)} T_{T_{x_2}^{L-2}(y_2)}^{-1})]_L\\
                       &\hspace{5.5pt}\vdots\\
                       &= L[T_{x_1}T_{x_2}^{-1}]_L + [T_{y_1} T_{y_2}^{-1}]_L,
\end{align*}
so $L[T_{x_1}T_{x_2}^{-1}]_L = 0$, as desired.
\end{proof}

\begin{lemma}
\label{lemma:splinv}
For $g \geq 3$ and $L \geq 2$, define $H = \HH_1(\Sigma_{g};\Z)$ and $H_L = \HH_1(\Sigma_{g};\Z / L\Z)$.  Then
$$(\wedge^3 H)_{\Sp_{2g}(\Z,L)} \cong \wedge^3 H_L \quad \text{and} \quad ((\wedge^3 H)/H)_{\Sp_{2g}(\Z,L)} \cong (\wedge^3 H_L)/H_L.$$
\end{lemma}
\begin{proof}
The natural maps $\wedge^3 H \rightarrow \wedge^3 H_L$ and $(\wedge^3 H)/H \rightarrow (\wedge^3 H_L)/H_L$ factor through
$(\wedge^3 H)_{\Sp_{2g}(\Z,L)}$ and $((\wedge^3 H)/H)_{\Sp_{2g}(\Z,L)}$ respectively, so it suffices to show
that in the indicated groups of coinvariants we have $L \cdot v = 0$ for all $v$.
Letting $S=\{a_1,b_1,\ldots,a_g,b_g\}$ be a symplectic basis for $H$, the groups $\wedge^3 H$ and
$(\wedge^3 H) / H$ are generated by $T:=\{\text{$x \wedge y \wedge z$ $|$ $x,y,z \in S$ distinct}\}$.  Consider $x \wedge y \wedge z \in T$.
It is enough to show that in the indicated groups of coinvariants we have $L(x \wedge y \wedge z)=0$.
Now, one of $x$, $y$, and $z$ must have algebraic intersection number $0$ with the other two terms.  Assume
that $x = a_1$ (the other cases are similar).  We thus have $y,z \in \{a_2,b_2,\ldots,a_g,b_g\}$.  There is then some
$\phi \in \Sp_{2g}(\Z,L)$ such that $\phi(b_1) = b_1 + L a_1 = b_1 + L x$ and such that $\phi(y)=y$ and $\phi(z)=z$.  We
conclude that in the indicated groups of coinvariants we have $b_1 \wedge y \wedge z = (b_1 + L x) \wedge y \wedge z$, so
$L(x \wedge y \wedge z) = 0$, as desired.
\end{proof}

\begin{proposition}
\label{proposition:torellicoinvariants}
Fix $g \geq 3$ and $L \geq 2$.  Set $H_L = \HH_1(\Sigma_{g,b};\Z/L)$.  If $L$ is odd, then
$$(\HH_1(\Torelli_{g,1};\Z))_{\Sp_{2g}(\Z,L)} \cong \wedge^3 H_L \quad \text{and} \quad (\HH_1(\Torelli_{g};\Z))_{\Sp_{2g}(\Z,L)} \cong (\wedge^3 H_L)/H_L,$$
while if $L$ is even, then
$$(\HH_1(\Torelli_{g,1};\Z))_{\Sp_{2g}(\Z,L)} \cong B_2(2g) \oplus \wedge^3 H_L \quad \text{and} \quad (\HH_1(\Torelli_{g};\Z))_{\Sp_{2g}(\Z,L)} \cong \overline{B}_2(2g) \oplus (\wedge^3 H_L)/H_L.$$
\end{proposition}

\begin{remark}
When $L$ is even, the decompositions of $(\HH_1(\Torelli_{g,1};\Z))_{\Sp_{2g}(\Z,L)}$ and $(\HH_1(\Torelli_{g};\Z))_{\Sp_{2g}(\Z,L)}$ 
in Proposition \ref{proposition:torellicoinvariants} are not natural with respect to the $\Sp_{2g}(\Z/L)$-action.
\end{remark}

\begin{proof}[Proof of Proposition \ref{proposition:torellicoinvariants}]
We will do the case of $\Torelli_{g,1}$; the case of $\Torelli_g$ is similar.  Let $H = \HH_1(\Sigma_g;\Z)$, so
$\HH_1(\Torelli_{g,1};\Z) \cong B_2(2g) \oplus \wedge^3 H$.  If $L$ is
odd, then by Lemma \ref{lemma:kill2torsion} the $2$-torsion $B_2(2g)$ goes to $0$ upon passing to the group of coinvariants.
Thus by Lemma \ref{lemma:splinv} we have
$$(\HH_1(\Torelli_{g,1};\Z))_{\Sp_{2g}(\Z,L)} \cong (\wedge^3 H)_{\Sp_{2g}(\Z)} \cong \wedge^3 H_L.$$
Now assume that $L$ is even.  Set 
$$K = \{\text{$x - g(x)$ $|$ $x \in \HH_1(\Torelli_{g,1};\Z)$, $g \in \Sp_{2g}(\Z,L)$}\},$$
so $(\HH_1(\Torelli_{g,1};\Z))_{\Sp_{2g}(\Z,L)} = \HH_1(\Torelli_{g,1};\Z) / K$.  Since $\Sp_{2g}(\Z,L)$ acts
trivially on $\HH_1(\Sigma_{g,1};\Z/2)$, the image of $K$ in $B_3(2g)$ under the Birman-Craggs-Johnson
homomorphism is $0$.  In particular, examining commutative diagram \eqref{eqn:johnsonhomobcghomodiagram} from
Theorem \ref{theorem:johnsonabel1}, we see
that $K \cap B_2(2g) = 0$.  Using Lemma \ref{lemma:splinv}, we thus conclude that
\begin{equation*}
(\HH_1(\Torelli_{g,1};\Z))_{\Sp_{2g}(\Z,L)} \cong B_2(2g) \oplus (\wedge^3 H)_{\Sp_{2g}(\Z)} \cong B_2(2g) \oplus \wedge^3 H_L. \qedhere
\end{equation*}
\end{proof}

\subsection{Extending abelian quotients of $\Torelli_{g,b}$ over $\Mod_{g,b}(L)$}
\label{section:extendingtorelli}

The last $3$ terms of the $5$ term exact sequence associated to the exact sequence
$$1 \longrightarrow \Torelli_{g,b} \longrightarrow \Mod_{g,b}(L) \longrightarrow \Sp_{2g}(\Z,L) \longrightarrow 1$$
are
\begin{equation}
\label{eqn:coinvexseq}
(\HH_1(\Torelli_{g,b};\Z))_{\Sp_{2g}(\Z,L)} \stackrel{j}{\longrightarrow} \HH_1(\Mod_{g,b}(L);\Z) \longrightarrow \HH_1(\Sp_{2g}(\Z,L)
\longrightarrow 0.
\end{equation}
In this section, we discuss two known results about the image of $j$.

As we proved in Proposition \ref{proposition:torellicoinvariants}, the group $(\HH_1(\Torelli_{g,b};\Z))_{\Sp_{2g}(\Z,L)}$
has two pieces, one coming from the Johnson homomorphism and the other from the Birman-Craggs-Johnson homomorphism (the
latter only existing if $L$ is even).  We begin by discussing the Johnson homomorphism $\tau$.  Analogues of the
Johnson homomorphism
``mod $L$'' were constructed independently by Broaddus-Farb-Putman, Perron, and Sato.  The following theorem
gives their main properties.
\begin{theorem}[{Broaddus-Farb-Putman, \cite[Example 5.3 and Theorem 5.8]{BroaddusFarbPutman}, Perron \cite{Perron},
Sato \cite{Sato}}]
\label{theorem:reljohnson}
Fix $g \geq 3$ and $L \geq 2$.  Set $H = \HH_1(\Sigma_{g};\Z)$ and
$H_L = \HH_1(\Sigma_{g};\Z/L)$.  There then exist homomorphisms
$\Mod_{g,1}(L) \rightarrow \wedge^3 H_L$ and $\Mod_g(L) \rightarrow (\wedge^3 H_L)/H_L$ that fit
into commutative diagrams of the form
$$\mbox{$
\begin{CD}
\Torelli_{g,1} @>{\tau}>> \wedge^3 H\\
@VVV                       @VVV \\
\Mod_{g,1}(L)  @>>>        \wedge^3 H_L
\end{CD}$}
\quad\quad \text{and} \quad\quad
\mbox{$
\begin{CD}
\Torelli_{g} @>{\tau}>> (\wedge^3 H)/H\\
@VVV                       @VVV \\
\Mod_{g}(L)  @>>>       (\wedge^3 H_L)/H_L
\end{CD}$}$$
Here the right hand vertical arrows are reduction mod $L$.
\end{theorem}
\noindent
An immediate consequence of this theorem is that the terms of $(\HH_1(\Torelli_{g,b};\Z))_{\Sp_{2g}(\Z,L)}$
given by the quotients of the Johnson homomorphisms inject into $\HH_1(\Mod_{g,b}(L);\Z)$.

We now turn to the Birman-Craggs-Johnson homomorphism $\sigma$.  A version of this was constructed
on $\Mod_{g,b}(2)$ by Sato.  However, he did not manage to get the entire image of $\sigma$.  Observe
that $B_3(2g)$ and $\overline{B}_3(2g)$ contain the ``constant'' subgroups $B_0(2g)$ and $\overline{B}_0(2g)$,
both of which are isomorphic to $\Z/2$.  Define
$$\sigma' : \Torelli_{g,1} \longrightarrow B_3(2g) / B_0(2g) \quad \text{and} \quad
\sigma' : \Torelli_g \longrightarrow \overline{B}_3(2g) / \overline{B}_0(2g)$$
to be the compositions of the Birman-Craggs-Johnson homomorphisms with the quotients by $B_0(2g)$ and
$\overline{B}_0(2g)$.  Sato proved the following theorem.
\begin{theorem}[{Sato \cite{Sato}}]
\label{theorem:relbcj}
Fix $g \geq 3$.  There then exist abelian groups $A$ and $\overline{A}$ together
with homomorphisms $\Mod_{g,1}(2) \rightarrow A$
and $\Mod_g \rightarrow \overline{A}$ and injections
$i : B_3(2g) / B_0(2g) \hookrightarrow A$ and 
$\overline{i} : \overline{B}_3(2g) / \overline{B}_0(2g) \hookrightarrow \overline{A}$ that fit into
commutative diagrams of the form
$$\mbox{$
\begin{CD}
\Torelli_{g,1} @>{\sigma'}>> B_3(2g) / B_0(2g)\\
@VVV                        @VV{i}V \\
\Mod_{g,1}(2)  @>>>         A
\end{CD}$}
\quad\quad \text{and} \quad\quad
\mbox{$
\begin{CD}
\Torelli_{g} @>{\sigma'}>> \overline{B}_3(2g) / \overline{B}_0(2g)\\
@VVV                        @VV{\overline{i}}V \\
\Mod_{g}(2)  @>>>          \overline{A}
\end{CD}$}$$
\end{theorem}

\begin{remark}
Sato also proves in \cite{Sato} that there does {\em not} exist an abelian group $B$ together with a homomorphism
$\Mod_{g,1}(2) \rightarrow B$ and an injection $j:B_3(2g) \hookrightarrow B$ such that the diagram
$$\begin{CD}
\Torelli_{g,1} @>{\sigma}>> B_3(2g) \\
@VVV                        @VV{j}V \\
\Mod_{g,1}(2)  @>>>         B
\end{CD}$$
commutes, and similarly for $\Mod_{g}(2)$.  For $L > 2$, Theorem \ref{theorem:modlabel} below
implies that a similar statement is true for $\Mod_{g,1}(L)$ and $\Mod_{g}(L)$.
\end{remark}

For $L$ even, we can restrict the homomorphisms given by Theorem \ref{theorem:relbcj} to $\Mod_{g,b}(L)$ and
obtain that the only possible subgroups of the pieces of $(\HH_1(\Torelli_{g,b};\Z))_{\Sp_{2g}(\Z,L)}$
given by the Birman-Craggs-Johnson homomorphisms that can go to $0$ in $\HH_1(\Mod_{g,b}(L);\Z)$ are
$B_0(2g)$ and $\overline{B}_0(2g)$.

Summing up, we have obtained the following lemma.
\begin{lemma}
\label{lemma:extendtorelli}
Fix $g \geq 3$ and $L \geq 2$ and $0 \leq b \leq 1$.  If $L$ is odd, then the map
$$(\HH_1(\Torelli_{g,b};\Z))_{\Sp_{2g}(\Z,L)} \rightarrow \HH_1(\Mod_{g,b}(L);\Z)$$         
from \eqref{eqn:coinvexseq} is injective.  If $L$ is even, then the kernel of this
map is isomorphic to either $0$ or $\Z/2$.  If the kernel is $\Z/2$, then it is
equal to $B_0(2g)$ if $b = 1$ and $\overline{B}_0(2g)$ if $b = 0$.
\end{lemma}

\begin{remark}
We will eventually prove that when $L$ is even, the kernel in Lemma \ref{lemma:extendtorelli} is $\Z/2$.
\end{remark}

\section{The moduli space of curves}
\label{section:modulispacecurves}

\subsection{Definitions}
\label{section:modulidef}

We now introduce the moduli space of curves $\Moduli_g$.  
For more information, see \cite{HainModuliTran, HarrisMorrison}.

Let $\Teich_g$ be {\em Teichm\"{u}ller space}, i.e.\ the space of pairs $(S,B)$ where $S$ is a genus $g$ Riemann
surface and $B$ is a standard basis for $\pi_1(S)$ (up to conjugacy).  The space $\Teich_g$ is a contractible
complex manifold and the group $\Mod_g$ acts on it holomorphically and properly discontinuously in the following way.  
If $\phi \in \Mod_g$ and $(S,B) \in \Teich_g$, then $\phi(B)$ is well-defined up to conjugacy and $\phi(S,B) = (S,\phi(B))$.

As a complex analytic space, $\Moduli_g = \Teich_g / \Mod_g$; however, for us it is important
to take into account its structure as an orbifold.
To simplify our statements, we will denote $\Mod_{g}$ and $\Moduli_g$ by $\Mod_{g}(1)$ and $\Moduli_g(1)$, respectively.
For $L \geq 1$, the space $\Moduli_g(L)$ is the orbifold
$(\Teich_g,\Mod_g(L))$.  Thus $\Moduli_g(L')$ is a covering orbifold for $\Moduli_g(L)$ whenever $L|L'$.  For
$L \geq 3$, the group $\Mod_{g}(L)$ acts freely on $\Teich_g$, so in these cases $\Moduli_g(L)$ is a
trivial orbifold.  Ignoring the orbifold structure for a moment, it is known that the complex analytic space
$\Teich_g / \Mod_g(L)$ is a quasiprojective variety for all $L$ (this was originally
proven by Baily \cite{BailySchottky}; see also \cite{DeligneMumford}).  We conclude that 
$\Moduli_g(L)$ is a quasiprojective orbifold with quasiprojective finite cover 
$\Moduli_g(L')$ for any $L' \geq 3$ such that $L|L'$.
The representation $\Mod_g \rightarrow \Sp_{2g}(\Z)$ is induced by the orbifold map $\Moduli_g \rightarrow \PPAV_g$
that takes a Riemann surface to its Jacobian.

\subsection{Some known results}
\label{section:moduliknown}

Just like in the case of $\Pic(\PPAV_g(L)$, our main tool for investigating $\Pic(\Moduli_g(L))$ will
be Lemma \ref{lemma:maintool}.  To verify the conditions of that lemma, we will need a number of results, 
the first of which are the following group-cohomology calculations.

\begin{theorem}[{Powell, \cite{PowellTorelli}}]
\label{theorem:h1mod}
For $g \geq 3$, we have $\HH_1(\Mod_g;\Z) = 0$.
\end{theorem}

\begin{theorem}[{Harer, \cite{HarerSecond}}]
\label{theorem:h2mod}
For $g \geq 4$, we have $\HH_2(\Mod_g;\Z) \cong \Z$.
\end{theorem}

\begin{remark}
Harer's paper \cite{HarerSecond} only considers the case $g \geq 5$ and contains an erroneous
torsion term.  The case $g = 4$ is essentially due to Pitsch \cite{PitschH2}, though his
paper only considers surfaces with boundary.  For more information, see the survey
\cite{KorkmazSurvey}, which also gives a very short proof (due to Harer) of Theorem \ref{theorem:h1mod}.
\end{remark}

\begin{theorem}[{Hain, \cite{HainTorelli}}]
\label{theorem:h1modl}
For $g \geq 3$ and $L \geq 2$, we have $\HH^1(\Mod_g(L);\Z) = 0$.
\end{theorem}

\begin{theorem}[{Putman, \cite{PutmanSecond}}]
\label{theorem:h2modl}
For $g \geq 5$ and $L \geq 2$, we have $\HH_2(\Mod_g(L);\Q) \cong \Q$.
\end{theorem}

To relate $\Moduli_g$ and $\PPAV_g$, we will need the following folklore result.  I do
not know a reference for it, so I include a proof.

\begin{lemma}
\label{lemma:modsp}
For $g \geq 4$, the map $\Mod_g \rightarrow \Sp_{2g}(\Z)$ induces an isomorphism
$\HH_2(\Mod_g;\Z) \cong \HH_2(\Sp_{2g}(\Z))$.
\end{lemma}
\begin{proof}
Theorem \ref{theorem:h2mod} says that $\HH_2(\Mod_g;\Z) \cong \Z$ and Theorem \ref{theorem:borel} says
that $\HH_2(\Sp_{2g}(\Z);\Z)$ has rank $1$, so it is enough to show that the map 
$\HH_2(\Mod_g;\Z) \rightarrow \HH_2(\Sp_{2g}(\Z);\Z)$ is a surjection.  The $5$ term exact sequence
coming from the short exact sequence
$$1 \longrightarrow \Torelli_g \longrightarrow \Mod_g \longrightarrow \Sp_{2g}(\Z) \longrightarrow 1$$
contains the segment 
$$\HH_2(\Mod_g;\Z) \longrightarrow \HH_2(\Sp_{2g}(\Z);\Z) \longrightarrow (\HH_1(\Torelli_g;\Z))_{\Sp_{2g}(\Z)}.$$
By Lemma \ref{lemma:kill2torsion}, we have $(\HH_1(\Torelli_g;\Z))_{\Sp_{2g}(\Z)} = 0$, and the lemma 
follows.
\end{proof}

\begin{remark}
One can obtain Theorem \ref{theorem:hsp} from the above results as follows.  First,
the map $\Mod_g \rightarrow \Sp_{2g}(\Z)$ is a surjection, so Theorem \ref{theorem:h1mod} implies that
$\HH_1(\Sp_{2g}(\Z);\Z) = 0$ for $g \geq 3$.  Second, one can combine Lemma \ref{lemma:modsp}
and Theorem \ref{theorem:h2mod} to deduce that $\HH_2(\Sp_{2g}(\Z);\Z) \cong \Z$ for $g \geq 4$.
\end{remark}

Next, let $\Line_g \in \Pic(\Moduli_g)$ be the determinantal bundle of the Hodge bundle, as in the
introduction.  As is shown in \cite[Corollary 17.4]{HainModuliTran}, the line bundle $\Line_g$ is the
pullback of $\PPAVLine_g \in \Pic(\PPAV_g)$.
We will need two properties of $\Line_g$.  First, it generates $\Pic(\Moduli_g)$.

\begin{theorem}[{Arbarello-Cornalba \cite{ArbarelloCornalba}, see also \cite[Theorem 17.3]{HainModuliTran}}]
\label{theorem:ppavlinegen2}
For $g \geq 4$, the class $c_1(\Line_g)$ generates $\HH^2(\Mod_g;\Z) \cong \Z$.
\end{theorem}

For the second, let $\Line_g(L) \in \Pic(\Moduli_g(L))$ be the pullback of $\Line_g$.  We will
need the following recent theorem of G.\ Farkas.

\begin{theorem}[{Farkas, \cite[Theorem 0.2]{Farkas}}]
\label{theorem:farkas}
For $g \geq 3$ and $L \geq 2$ such that $L$ is even, there exists some $\theta_g(L) \in \Pic(\Moduli_g(L))$
such that $4 \theta_g(L) - \Line_g(L)$ is torsion.
\end{theorem}
\begin{remark}
The statement of \cite[Theorem 0.2]{Farkas} differs from Theorem \ref{theorem:farkas} in three ways.  First, the
line bundle $\theta_{\text{null}}$ given there actually lives in Picard group of the moduli space $\SpinModuli^{+}_g$ of curves with even spin
structures.  For $L$ even, the orbifold $\Moduli_g(L)$ is a finite cover of $\SpinModuli^{+}_g$, so the line bundle
whose existence is asserted by Theorem \ref{theorem:farkas} can be obtained by pulling back $\theta_{\text{null}}$.  Second, the statement
of \cite[Theorem 0.2]{Farkas} actually refers to a compactification of $\SpinModuli^{+}_g$ and contains a number
of extraneous boundary terms that vanish when restricted to the open moduli space.  Third, there is no
mention of torsion in \cite[Theorem 0.2]{Farkas} because Farkas is making a computation in the rational
Picard group.
\end{remark}

\subsection{The proofs of Theorems \ref{theorem:mg} and \ref{theorem:mzgen} and \ref{theorem:mh2}}
\label{section:proofsm}

We wish to apply Lemma \ref{lemma:maintool} with $(X,G) = (\Teich_g,\Mod_{g})$ and $G' = \Mod_g(L)$
and $\lambda = \Line_g$ and
$$n = \begin{cases}
1 & \text{if $L$ is odd,} \\
4 & \text{if $L$ is even.}
\end{cases}$$
The results of \S \ref{section:modulidef}--\ref{section:moduliknown}
show that all the conditions of Lemma \ref{lemma:maintool} are satisfied except possibly
the condition that the image of $\HH_2(\Mod_g(L);\Z)$ in $\HH_2(\Mod_{g};\Z) \cong \Z$ is
$m \Z$ for some $m$ with $1 \leq m \leq n$, which we now prove (we remark that Lemma \ref{lemma:maintool} will
then imply that $m=n$).

By Lemma \ref{lemma:modsp}, it is enough to show that the image of $\HH_2(\Mod_g(L);\Z)$
in $\HH_2(\Sp_{2g}(\Z);\Z) \cong \Z$ is $m \Z$ for some $m$ with $1 \leq m \leq n$.  The $5$-term exact
sequence associated with the short exact sequence
$$1 \longrightarrow \Torelli_g \longrightarrow \Mod_g(L) \longrightarrow \Sp_{2g}(\Z,L) \longrightarrow 1$$
contains the segment 
$$\HH_2(\Mod_g(L);\Z) \longrightarrow \HH_2(\Sp_{2g}(\Z,L);\Z) \longrightarrow (\HH_1(\Torelli_g;\Z))_{\Sp_{2g}(\Z,L)}
\stackrel{i}{\longrightarrow} \HH_1(\Mod_g(L);\Z).$$
Assume first that $L$ is odd.  Lemma \ref{lemma:extendtorelli} says that $\Ker(i) = 0$, so the map
$\HH_2(\Mod_g(L);\Z) \rightarrow \HH_2(\Sp_{2g}(\Z,L);\Z)$ is surjective.  Theorem \ref{theorem:splh2} says that
the map $\HH_2(\Sp_{2g}(\Z,L);\Z) \rightarrow \HH_2(\Sp_{2g}(\Z);\Z)$ is surjective, so we conclude that
the map $\HH_2(\Mod_g(L);\Z) \rightarrow \HH_2(\Sp_{2g}(\Z);\Z)$ is surjective, as desired.

Assume now that $L$ is even.  Lemma \ref{lemma:extendtorelli} says that $\Ker(i)$ is either $0$ or
$\Z/2$, so the cokernel of the map $\HH_2(\Mod_g(L);\Z) \rightarrow \HH_2(\Sp_{2g}(\Z,L);\Z)$ is either
$0$ or $\Z/2$.  Theorem \ref{theorem:splh2} says that the image of
$\HH_2(\Sp_{2g}(\Z,L);\Z)$ in $\HH_2(\Sp_{2g}(\Z);\Z) \cong \Z$ is $2 \Z$, so we conclude that
the image of the map $\HH_2(\Mod_g(L);\Z) \rightarrow \HH_2(\Sp_{2g}(\Z);\Z)$ is either $2 \Z$ or $4 \Z$ (and hence by
Lemma \ref{lemma:maintool} it is $4 \Z$), as desired.

\subsection{The abelianization of $\Mod_g(L)$}

The following theorem subsumes Theorem \ref{theorem:modlabel1} from the introduction.

\begin{theorem}
\label{theorem:modlabel}
Fix $g \geq 5$ and $0 \leq b \leq 1$ and $L \geq 2$ such that $4 \nmid L$.  Set $H_L = \HH_1(\Sigma_g;\Z/L)$.  
There is then a short exact sequence
$$0 \longrightarrow K_{g,b} \longrightarrow \HH_1(\Mod_{g,b}(L);\Z) \longrightarrow \HH_1(\Sp_{2g}(\Z,L);\Z)
\longrightarrow 0,$$
where
$$K_{g,1} \cong \wedge^3 H_L \quad \text{and} \quad K_{g} \cong (\wedge^3 H_L) / H_L$$
if $L$ is odd and
$$K_{g,1} \cong (B_2(2g) / B_0(2g)) \oplus \wedge^3 H_L \quad \text{and} \quad 
K_{g} \cong (\overline{B}_2(2g) / \overline{B}_0(2g)) \oplus (\wedge^3 H_L) / H_L$$
if $L$ is even.
\end{theorem}
\begin{proof}
The $5$-term exact sequence coming from the short exact sequence
$$1 \longrightarrow \Torelli_{g,b} \longrightarrow \Mod_{g,b}(L) \longrightarrow \Sp_{2g}(\Z,L) \longrightarrow 1$$
contains the subsequence
$$(\HH_1(\Torelli_{g,b};\Z))_{\Sp_{2g}(\Z,L)} \stackrel{i_{g,b}}{\longrightarrow} \HH_1(\Mod_{g,b}(L);\Z) \longrightarrow
\HH_1(\Sp_{2g}(\Z,L);\Z) \longrightarrow 0.$$
Setting
$$K_{g,b} = (\HH_1(\Torelli_{g,b};\Z))_{\Sp_{2g}(\Z,L)} / \Ker(i_{g,b}),$$
it is enough to prove the indicated description of $K_{g,b}$.  If $L$ is odd, then we showed
in Lemma \ref{lemma:extendtorelli} that $i_{g,b}$ is injective, so the theorem follows from
Proposition \ref{proposition:torellicoinvariants}.  If $L$ is even and $b=0$, then 
in the proofs of Theorems \ref{theorem:mzgen} and \ref{theorem:mh2} in \S \ref{section:proofsm}, we showed
that $\Ker(i_{g,b}) \cong \Z/2$.  The theorem thus follows in this case from Proposition \ref{proposition:torellicoinvariants}
and Lemma \ref{lemma:extendtorelli}.  If $L$ is even and $b=1$, then we can use the case $b=0$ and the commutative diagram
$$\begin{CD}
\HH_2(Sp_{2g}(\Z,L);\Z) @>>> (\HH_1(\Torelli_{g,1};\Z))_{\Sp_{2g}(\Z,L)} @>{i_{g,1}}>> \HH_1(\Mod_{g,1}(L);\Z) \\
@VV{=}V                      @VVV                  @VVV \\
\HH_2(\Sp_{2g}(\Z,L);\Z) @>>> (\HH_1(\Torelli_{g};\Z))_{\Sp_{2g}(\Z,L)}  @>{i_{g,0}}>> \HH_1(\Mod_{g}(L);\Z)
\end{CD}$$
arising from the naturality of the $5$-term exact sequence to deduce that $\Ker(i_{g,1})$ must be nonzero.
Just like before, the theorem now follows from Proposition \ref{proposition:torellicoinvariants}
and Lemma \ref{lemma:extendtorelli}.
\end{proof}

\noindent
Department of Mathematics\\
Rice University, MS 136 \\
6100 Main St.\\
Houston, TX 77005\\
E-mail: {\tt andyp@rice.edu}

\end{document}